\documentclass[11pt]{amsart}

\usepackage {enumerate}
\usepackage{setspace}
\usepackage{caption}
\usepackage{hyperref}
\usepackage{esint}
\usepackage{graphicx,amssymb,amsmath,amsthm,mathrsfs}

\usepackage{color}
\usepackage{CJK}
\usepackage{setspace}

\usepackage{lipsum}

\newtheorem{theorem}{Theorem}[section]
\newtheorem{lemma}[theorem]{Lemma}
\newtheorem{corollary}[theorem]{Corollary}
\newtheorem{proposition}[theorem]{Proposition}
\newtheorem{example}[theorem]{Example}

\numberwithin{equation}{section}

\theoremstyle {definition}
\newtheorem{definition}[theorem]{Definition}
\newtheorem{remark}[theorem]{Remark}
\usepackage{epstopdf}
\usepackage[margin=1.1in]{geometry}

\usepackage[latin1]{inputenc}
\usepackage[T1]{fontenc}
\usepackage[english]{babel}
\usepackage{times}

\allowdisplaybreaks

\DeclareMathOperator{\area}{area}

\DeclareMathOperator{\id}{id}
\DeclareMathOperator{\Ric}{Ric}

\begin{document}

\title{Positive mass theorems of ALF and ALG manifolds}

\author{Peng Liu}
\address [Peng Liu] {Key Laboratory of Pure and Applied Mathematics, School of Mathematical Sciences, Peking University, Beijing, 100871, P.\ R.\ China}
\email{1801110011@pku.edu.cn}
\author{Yuguang Shi}
\address [Yuguang Shi] {Key Laboratory of Pure and Applied Mathematics, School of Mathematical Sciences, Peking University, Beijing, 100871, P.\ R.\ China}
\email{ygshi@math.pku.edu.cn}
\thanks{Yuguang Shi is partially supported by National Key R$\&$D Program of China, Grant NO.11731001.}

\author{Jintian Zhu}
\address [Jintian Zhu]   {Key Laboratory of Pure and Applied Mathematics, School of Mathematical Sciences, Peking University, Beijing, 100871, P. R. China}
\email{zhujt@pku.edu.cn}

\renewcommand{\subjclassname}{
 \textup{2010} Mathematics Subject Classification}

\date{\today}

\begin{abstract}
In this paper, we want to prove positive mass theorems for ALF and ALG manifolds with model spaces $\mathbb R^{n-1}\times \mathbb S^1$ and $\mathbb R^{n-2}\times \mathbb T^2$ respectively in dimensions no greater than $7$ (Theorem \ref{ALFPMT0}). { Different from the compatibility condition for spin structure in \cite[Theorem 2]{minerbe2008a}, we show that some type of incompressible condition for $\mathbb S^1$ and $\mathbb T^2$ is enough to guarantee the nonnegativity of the mass.} As in the asymptotically flat case, we reduce the desired positive mass theorems to those ones concerning non-existence of positive scalar curvature metrics on closed manifolds coming from generalize surgery to $n$-torus. { Finally, we investigate certain fill-in problems and obtain an optimal bound for total mean curvature of admissible fill-ins for flat product $2$-torus $\mathbb S^1(l_1)\times \mathbb S^1(l_2)$.}

\end{abstract}	
\markboth{Liu Peng Shi Yuguang, Zhu Jintian  }{Positive mass theorems of ALF and ALG manifolds}
\maketitle
\section{Introduction}

In the past a few decades, the positive mass theorem of asymptotically flat (AF) manifolds is one of basic results both in geometry of scalar curvature and the General Relativity (see \cite{SY2017}). From geometry as well as the String Theory, people have great intentions to generalize the positive mass theorem to more general situations { and many attempts have been done in this aspect} (see \cite{brill1991negative}, \cite{minerbe2008a}, \cite{dai2004a} and references therein).
{ In this paper, we would like to make a further discussion on positive mass theorems for ALF and ALG manifolds. These manifolds arise naturally from the research for asymptotically loacally flat gravitational instantons (see \cite{Gabor2006} for details), which have attracted great attention of mathematicians.

Although ALF and ALG manifolds should have more general definitions, we will only focus on the following special cases that are enough for our purpose.
Let $n\geq 4$. A complete $n$-manifold $(M,g)$ is called {\it ALF of Schwarzschild type} with decay order $\mu$ if there is a compact subset $K$ such that $M-K$ is diffeomorphic to $(\mathbb R^{n-1}-\mathbb B)\times \mathbb S^1$ and that the pullback metric is in the form
$$
g=\left(1+\frac{m}{2r^{n-3}}\right)^{\frac{4}{n-3}}(\mathrm dr^2+r^2\mathrm d\phi^2)+l^2\mathrm d\theta^2+\sigma
$$
with
$$\sum^2_{k=0}\sum_{|\alpha|=k}r^k|\partial^\alpha \sigma| =O(r^{-\mu}),\quad\text{as}\quad r\to\infty,\quad \mu>n-3,$$
where $m$, $l$ and $\mu$ are certain constants, $(r,\phi)$ is the polar coordinate system of $\mathbb R^{n-1}$ and $\mathrm d\theta$ is the length element of the unit circle.
In above definition, the constant $m$ is called the {\it mass} of the ALF manifold $(M,g)$. We have to mention that Minerbe gave a definition of mass in \cite{minerbe2008a} for more general ALF manifolds and our mass actually coincides with his one.

In dimension three, we can also define {\it an ALF manifold of conical type} with decay order $\mu$, i.e. a complete $3$-manifold $(M,g)$ with compact subset $K$ such that $M-K$ is diffeomorphic to $(\mathbb R^{2}-\mathbb B)\times \mathbb S^1$ and that the pullback metric is in the form
$$
g=\mathrm dr^2+\beta^2 r^2\mathrm d\phi^2+l^2\mathrm d\theta^2+\sigma
$$
with
$$\sum^2_{k=0}\sum_{|\alpha|=k}r^k|\partial^\alpha \sigma| =O(r^{-\mu}),\quad\text{as}\quad r\to\infty,\quad \mu>0.$$
The constant $2\pi\beta$ is called the cone angle of the ALF $(M,g)$. After replacing $(\mathbb R^{n-1}-\mathbb B)\times \mathbb S^1$ by $(\mathbb R^{n-2}-\mathbb B)\times \mathbb T^2$ and $l^2\mathrm d\theta^2$ by a flat metric on $\mathbb T^2$
we can define ALG manifolds of Schwarzschild type and conical type with decay order $\mu$ in a similar way.

From the philosophy of the positive mass theorem for AF manifolds, one would like to deduce non-negative mass $m$ or cone angle no greater than $2\pi$ from the assumption of non-negative scalar curvature. Unfortunately, this is not always the case since} the Euclidean Reissnet-Nordstrom metric on $\mathbb R^2\times \mathbb S^2$ (see \cite{brill1991negative}) serves as a counterexample.

\begin{example}\label{NTmetric}
	Let $M^4=\mathbb{R}^2 \times \mathbb{S}^2$ and $m$, $q$  be two constants with $m<0$ and $q>0$. Define

$$ g=\left(1-\frac{2 m}{r}-\frac{q^{2}}{r^{2}}\right) d \theta^{2}+\left(1-\frac{2 m}{r}-\frac{q^{2}}{r^{2}}\right)^{-1} d r^{2}+r^{2} d \omega^{2},
$$
with
$$r\geq r_+=m+\sqrt{m^2+q^2},$$
where $d \omega^{2}$ is the standard round metric on the unit sphere $\mathbb{S}^2$, $\theta$ is in a circle with length of  $\tau$ denoted by $\mathbb{S}^1_\tau \subset \mathbb{R}^2$. We may choose  a suitable $\tau$ so that $g$ is smooth when $r=r_+$, hence,  $\tau$ depends only on $m$ and $q$, then $g$ is an ALF metric on $M$ with being asymptotic to the geometry of $\mathbb{R}^3 \times \mathbb{S}^1$ and with scalar curvature $R=0$, while its mass (in Minerbe's definition) equals to $2m$ that is negative.
\end{example}

{
From this point of view, some further conditions are needed to guarantee the desired positive mass theorem. In this aspect, Minerbe came up with two types of conditions in \cite{minerbe2008a} to make the positive mass theorem become true. With curvature condition enhanced, he showed that an ALF manifolds with non-negative Ricci curvature has a non-negative mass. After this he also proved the same thing for spin ALF manifolds with non-negative scalar curvature whose spin structure at infinity is induced from $(\mathbb R^{n-1}-\mathbb B)\times \mathbb S^1$.
Compared with Minerbe's work, we choose a very different solution based on the topological structure of ALF and ALG manifolds. In our understanding, the key feature of Example \ref{NTmetric} is that the $\mathbb S^1$-factor at infinity can shrink to a point and these counterexamples can be avoided by some type of incompressible conditions of $\mathbb S^1$ or $\mathbb T^2$-factor at infinity.}
In fact, we can prove the following

\begin{theorem}\label{ALFPMT0}
Let $(M^n,g)$ be an ALF or ALG  manifold with scalar curvature $R_g\geq 0$ and $n\leq 7$. Denote $\mathbb S^1$ and $\mathbb T^2$ to be the $\mathbb T^k$-factor of $M$ at infinity. Then we have the following:
\begin{itemize}
\item if $M$ is ALF of Schwarzschild type with decay order $\mu>2n-6$ and homotopy non-trivial $\mathbb S^1$-factor or ALG of Schwarzschild type with decay order $\mu>2n-8$ such that the map $i_*:\pi_1(\mathbb T^2)\to \pi_1(M)$ is injective, then the mass $m$ of $(M,g)$ is nonnegative. Moreover, the mass is zero if and only if it is isometric to $\mathbb R^{n-1}\times \mathbb S^1(l)$ (ALF case) or $\mathbb R^{n-2}\times \mathbb T^2$ (ALG case) with a flat $\mathbb T^2$;
\item if $M$ is ALF of conical type with decay order $\mu>1$ and homotopy non-trivial $\mathbb S^1$-factor or ALG of conical type with decay order $\mu>1$ such that the map $i_*:\pi_1(\mathbb T^2)\to \pi_1(M)$ is injective, then the cone angle $2\pi\beta$ of $(M,g)$ is no greater than $2\pi$, where the equality holds if and only if $M$ is isometric to $\mathbb R^2\times \mathbb S^1(l)$ (ALF case) or $\mathbb R^2\times \mathbb T^2$ (ALG case) with a flat $\mathbb T^2$.
\end{itemize}
\end{theorem}

{ About the proof the basic idea is to reduce above theorem to non-existence results of positive scalar curvature (PSC) metrics on certain closed manifolds.} Compared to the AF case, conformal deformation of metrics does not preserve ALF and ALG manifolds of Schwarzschild or conical type this time due to the $\mathbb{T}^k$-factor at the infinity of ALF and ALG manifolds and we have to work out a new way to overcome this difficulty. { Based on quasi-spherical metrics (see \cite{Bartnik1993} and \cite{ST2002}) and gluing method, we can always modify the original manifold to a piecewise smooth one with PSC in the distribution sense (see \cite{Miao2002}), whose infinity is isometric to that of standard product $\mathbb R^{n-k}\times \mathbb T^k$, under the assumption that the mass is negative or that the cone angle is greater than $2\pi$. After gluing opposite faces of a large cube, we reduce our goal to rule out the possibility of any PSC metric on closed manifolds coming from general surgery to $\mathbb T^n$ along the tube neighborhood of $\mathbb T^k$ (refer to Section \ref{Sec: general surgery} for more details).

Incompressible conditions are very common as a known obstacle for PSC metrics. It follows directly from the Gauss-Bonnet formula and the classification of closed surfaces that an orientable closed surface admits no PSC metric if it has a non-contractible closed curve. In dimension three, Schoen and Yau \cite{SY1979b} showed that an orientable closed $3$-manifold $M$ with an immersed $2$-torus $\Sigma$ such that $i_*:\pi_1(\Sigma)\to \pi_1(M)$ is injective cannot admit any PSC metric. Similar incompressible conditions are adapted in our investigation on PSC metrics and general surgeries to $n$-torus. With the notation in Section \ref{Sec: general surgery}, we can show
\begin{theorem}\label{Thm: main 2.1}
Denote $\gamma$ to be a closed curve in $\partial \Omega$ such that $\gamma$ is homotopic to $\mathbb S^1_n$ in $\mathbb T^n-\Omega$. If the curve $\phi_2(\gamma)$ is homotopically non-trivial in $M_2-\phi_2(\Omega)$ for some $(M_2,\phi_2)$ in $\mathcal C_\Omega$,
then $(\mathbb T^n,i)\sharp (M_2,\phi_2)$ admits no PSC metric for $n\leq 7$. Moreover, if $g$ is a smooth metric on $(\mathbb T^n,i)\sharp (M_2,\phi_2)$ with non-negative scalar curvature, then it is flat.
\end{theorem}
\begin{theorem}\label{Thm: main 2.2}
Denote $\Sigma$ to be a $2$-torus in $\partial \Omega$ such that $\Sigma$ is homotopic to $\mathbb S^1_{n-1}\times\mathbb S^1_n$ in $\mathbb T^n-\Omega$. If the inclusion map
$$i_*:\pi_1(\phi_2(\Sigma))\to \pi_1(M_2-\phi_2(\Omega))$$
 is injective for some $(M_2,\phi_2)$ in $\mathcal C_\Omega$,
then $(\mathbb T^n,i)\sharp (M_2,\phi_2)$ admits no PSC metric for $n\leq 7$. Moreover, if $g$ is a smooth metric on $(\mathbb T^n,i)\sharp (M_2,\phi_2)$ with non-negative scalar curvature, then it is flat.
\end{theorem}
Roughly speaking, the incompressible conditions forbid the existence of submanifolds with positive Yamabe invariant coming from Schoen-Yau's dimension reduction argument. To show this we need to maximize the use of intersection structure of the underlying manifolds, which turns out to be a key feature of our present work.
} Relationships between positive scalar curvature and general surgeries were first researched in \cite{SY1979} as well as \cite{GL80}. From this aspect, Theorem \ref{Thm: main 2.1} and \ref{Thm: main 2.2} have their own interest.

Another motivation to consider the positive mass theorems on ALF and ALG manifolds is the fill-in problem of non-negative scalar curvature raised by M. Gromov (see \cite{gromov2020lectures}). { Given an orientable closed Riemannian manifold $(\Sigma,\gamma)$, an {\it admissible fill-in} of $(\Sigma,\gamma)$ is a compact orientable Riemannian manifold $(\Omega,g)$ with non-negative scalar curvature and mean convex boundary such that the boundary is isometric to $(\Sigma,\gamma)$. The supremum total mean curvature of $(\Sigma,\gamma)$ with respect to admissible fill-ins is defined to be
$$
\Lambda_+(\Sigma,\gamma)=\sup\left\{\int_{\partial\Omega}H_{\partial\Omega}\,\mathrm d\sigma_g:\text{ $(\Omega,g)$ is an admissible fill-in of $(\Sigma,\gamma)$}\right\},
$$
where $H_{\partial\Omega}$ is the mean curvature of boundary $\partial\Omega$ with respect to the outer unit normal and $\mathrm d\sigma_g$ is the area element of $\partial\Omega$.
Some discussions on the quantity $\Lambda_+$ have been made in \cite{shi2020total} and \cite{yuguang2019on} when $\Sigma$ is an $n$-sphere. As applications of Theorem \ref{Thm: main 2.1} and \ref{Thm: main 2.2}, we make a further investigate on fill-in problems when $\Sigma$ is diffeomorphic to $\mathbb S^p\times \mathbb S^1$ or $\mathbb S^p\times \mathbb T^2$. As a result, we can obtain the following theorems.
\begin{theorem}\label{Thm: fill-in 1}
Let $\Sigma_0$ be a convex hypersurface or curve in the Euclidean space $\mathbb R^n$ with total mean curvature $T_0$. If $(\Omega,g)$ is an admissible fill-in of the product manifold $\Sigma_0\times \mathbb S^1(l)$ such that the circle component is homotopically non-trivial in $\Omega$, then for $n\leq 6$ it holds
\begin{equation}\label{Eq: total mean curvature 1}
\int_{\partial\Omega}H_{\partial\Omega}\,\mathrm d\sigma_g\leq 2\pi l T_0,
\end{equation}
where $H_{\partial\Omega}$ is the mean curvature of $\partial\Omega$ with respect to the unit outer normal and $\mathrm d\sigma_g$ is the area element of $\partial\Omega$ with the induced metric. If $(\Omega,g)$ is an admissible fill-in of $\Sigma_0\times \mathbb S^1(l)$ with the equality in \eqref{Eq: total mean curvature 1}, then $(\Omega,g)$ is static with vanishing scalar curvature.
\end{theorem}

\begin{theorem}\label{Thm: fill-in 2}
Let $\Sigma_0$ be a convex hypersurface or curve in the Euclidean space $\mathbb R^n$ with total mean curvature $T_0$. If $(\Omega,g)$ is an admissible fill-in of the product manifold $\Sigma_0\times \mathbb T^2$ with a flat $\mathbb T^2$ such that the inclusion map $i_*:\pi_1(\mathbb T^2)\to \pi_1(\Omega)$ is injective, then for $n\leq 5$ it holds
\begin{equation}\label{Eq: total mean curvature 2}
\int_{\partial\Omega}H_{\partial\Omega}\,\mathrm d\sigma_g\leq T_0 \area(\mathbb T^2),
\end{equation}
where $H_{\partial\Omega}$ is the mean curvature of $\partial\Omega$ with respect to the unit outer normal and $\mathrm d\sigma_g$ is the area element of $\partial\Omega$ with the induced metric. If $(\Omega,g)$ is an admissible fill-in of $\Sigma_0\times \mathbb T^2$ with the equality in \eqref{Eq: total mean curvature 2}, then $(\Omega,g)$ is static with vanishing scalar curvature.
\end{theorem}
\begin{remark}
The estimates \eqref{Eq: total mean curvature 1} and \eqref{Eq: total mean curvature 2} are optimal. To see this, we let $\Omega_0$ be the region enclosed by $\Sigma_0$ in the Euclidean space $\mathbb E^n$, then the admissible fill-in $\Omega_0\times \mathbb S^1(l)$ and $\Omega_0\times \mathbb T^2$ makes \eqref{Eq: total mean curvature 1} and \eqref{Eq: total mean curvature 2} become an equality respectively.
\end{remark}
\begin{remark}
Taking $\mathbb R^{n-k}\times \mathbb T^k$ as a model space, our results yield the positivity of Brown-York mass given by
$$
m_{BY}(\Omega,g)=T_0\cdot\area(\mathbb T^k)-\int_{\partial\Omega}H_{\partial\Omega}\,\mathrm d\sigma_g.
$$ 
Generally, the incompressible conditions are necessary since the region $\Sigma_0\times \mathbb D^2(l)$ would make the Brown-York mass negative when $l$ is small enough.
\end{remark}

The theorems above come from an application of quasi-spherical metrics as in \cite{yuguang2019on}, but the quasi-spherical equation now appears to be a degenerated parabolic equation, which leads to extra trouble in analyzing the behavior of its solution. We get over this difficulty by a clever use of the lifting property from $\mathbb T^k$-factors to remove the degeneracy (refer to the proof of Theorem \ref{Thm: fill-in 1} for details).

From above theorem we can also establish an optimal estimate for supremum total mean curvature of product flat torus $\mathbb S^1(l_1)\times \mathbb S^1(l_2)$.
\begin{corollary}\label{Cor: fill-in 2-torus}
Let $(\Sigma,\gamma)$ be the flat $2$-torus $\mathbb S^1(l_1)\times \mathbb S^1(l_2)$ with $l_1\leq l_2$. Then it holds
$$
\Lambda_+(\Sigma,\gamma)= 4\pi^2 l_2.
$$
Moreover, if $(\Omega,g)$ is an admissible fill-in of $(\Sigma,\gamma)$ whose boundary has total mean curvature $4\pi^2l_2$, then it is static with vanishing scalar curvature.
\end{corollary}

In their paper \cite{shi2020total}, the second author and his collaborators built the equivalence between the positive mass theorem of AF $n$-manifolds and the optimal estimate $\Lambda_+(\mathbb S^{n-1},\gamma_{std})=(n-1)\omega_{n-1}$ for the standard round sphere. It is also appealing to show the equivalence between above optimal estimate and the positive mass theorem for general ALF $3$-manifolds.
{ Since all flat metrics on $\mathbb T^2$ form a connected space, we can further obtain the finiteness of supremum total mean curvature for flat $2$-toruses. Namely, we have
\begin{corollary}\label{Cor: total mean curvature flat torus}
Let $(\Sigma,\gamma)$ be a flat $2$-torus. Then $\Lambda_+(\Sigma,\gamma)<+\infty$.
\end{corollary}}

The rest of the paper is organized as follows. In section \ref{Sec: general surgery}, we define the general surgery considered in this paper and devote a proof for Theorem \ref{Thm: main 2.1} and \ref{Thm: main 2.2}. The proof of these theorems comes from a careful analysis on intersections of the submanifold given by dimension reduction argument and an $(n-1)$-torus $S$. For Theorem \ref{Thm: main 2.1} the homotopy non-trivial $\mathbb S^1$ helps us to show the manifold after surgery is a Schoen-Yau-Schick (SYS) manifold defined in \cite{Gromov18} and so the result follows quickly. For Theorem \ref{Thm: main 2.2} the injectivity of map $i_*$ is used to show that the intersection is in the image of Hurewicz homomorphism $\pi_2(S)\to H_2(S,\mathbf Z)$, which leads to a contradiction to the homology non-trivial property of this intersection from the dimension reduction argument. In Section \ref{Sec: PMT}, we give a proof for Theorem \ref{ALFPMT0}. We show how to reduce positive mass theorems of ALF and ALG manifolds to Theorem \ref{Thm: main 2.1} and \ref{Thm: main 2.2} based on quasi-spherical metrics and gluing method. A rigidity argument with a use of Ricci flow is included in this section as well. In Section \ref{Sec: fill-in}}, we investigate the fill-in problem for manifolds in the form of $\mathbb S^p\times \mathbb S^1$ or $\mathbb S^p\times \mathbb T^2$. Proofs for Theorem \ref{Thm: fill-in 1}, Theorem \ref{Thm: fill-in 2}, Corollary \ref{Cor: fill-in 2-torus} and Corollary \ref{Cor: total mean curvature flat torus} are given in that section.

 \medskip
 {\it Acknowledgements.} We would like to thank Professor Shing-Tung Yau for drawing our attention to positive mass theorems on manifolds with general asymptotic structure at the infinity. We are also grateful to Dr. Chao Li for many inspiring discussions with the second author on relationships between positive mass theorem for ALF manifolds and non-negative fill-in problems.

\section{General surgery to $n$-torus and positive scalar curvature}\label{Sec: general surgery}
Let $\Omega$ be an orientable compact $n$-manifold. We collect all orientable closed $n$-manifolds containing $\Omega$ as a subregion and denote
$$
\mathcal C_\Omega=\left\{(M,\phi)\left|\begin{array}{c}\text{$M$ is an orientable compact $n$-manifold}\\
\text{$\phi:\Omega\to M$ is an embedding}
\end{array}\right.\right\}.
$$
Given any pairs $(M_1,\phi_1)$ and $(M_2,\phi_2)$ in $\mathcal C_\Omega$, we can define the gluing space
of $(M_1,\phi_1)$ and $(M_2,\phi_2)$ along $\Omega$ by
$$
M=\left(M_1-\Omega_1\right)\sqcup_{\Phi}\left(M_2-\Omega_2\right),
$$
where $\Omega_i=\phi_i(\Omega)$ for $i=1,2$ and
$
\Phi=\phi_2\circ\phi_1^{-1}:\partial \Omega_1\to \partial \Omega_2.
$
For convenience, we denote
$$M=(M_1,\phi_1)\sharp_\Omega (M_2,\phi_2) $$
and we will further omit the symbol $\Omega$ if it is clear from the context. Since the construction above reduces to connected sum if $\Omega$ is a $n$-ball, we also call $M$ the manifold from a general surgery between $(M_1,\phi_1)$ and $(M_2,\phi_2)$.

It is interesting to consider the relation between PSC metrics and general surgeries. In this aspect, the following results are well-known.

\begin{theorem}[Schoen-Yau \cite{SY1979}]
If $i:\Omega\to \mathbb T^n$ is a ball in $\mathbb T^n$, then $(\mathbb T^n,i)\sharp (M_2,\phi_2)$ admits no PSC metric for any $(M_2,\phi_2)$ in $\mathcal C_\Omega$.
\end{theorem}

\begin{theorem}[Schoen-Yau \cite{SY1979}, Gromov-Lawson \cite{GL1980}]\label{Thm: codim 3 surgery}
If $\Omega$ is a tube neigborhood of a submanifold in it with codimention no less than three, then $(M_1,\phi_1)\sharp(M_2,\phi_2)$ admits a PSC metric for any $(M_1,\phi_1)$ and $(M_2,\phi_2)$ in $\mathcal C_\Omega$ if both $M_1$ and $M_2$ admit a PSC metric.
\end{theorem}

In this paper, we consider similar problems for manifolds from general surgeries to $n$-torus $\mathbb T^n$. Denote $\mathbb S^1_n$ to be the last circle component of $\mathbb T^n$ and let $i:\Omega\to \mathbb T^n$ be a tube neighborhood of $\mathbb S^1_n$.

We have the following result.

\begin{theorem}\label{Thm: homotopy nontrivial no PSC}
Denote $\gamma$ to be a closed curve in $\partial \Omega$ such that $\gamma$ is homotopic to $\mathbb S^1_n$ in $\mathbb T^n-\Omega$. Suppose that the curve $\phi_2(\gamma)$ is homotopically non-trivial in $M_2-\phi_2(\Omega)$ for some $(M_2,\phi_2)$ in $\mathcal C_\Omega$.
Then $(\mathbb T^n,i)\sharp (M_2,\phi_2)$ admits no PSC metric for $n\leq 7$. Moreover, if $g$ is a smooth metric on $(\mathbb T^n,i)\sharp (M_2,\phi_2)$ with non-negative scalar curvature, then it is flat.
\end{theorem}

Our theorem indicates more concrete examples of closed manifold that admits no PSC metric such as
\begin{equation}\label{Eq: homotopy non-trivial example}
(\mathbb T^{n-1}-B)\times \mathbb S^1_n\sqcup_\phi \mathbb S^{n-2}\times (\Sigma_h-D),\quad \phi=\phi_1\times \phi_2:\partial B\times \mathbb S^1\to \mathbb S^{n-2}\times \partial D,
\end{equation}
where $\Sigma_h$ is an orientable closed surface with positive genus $h$ and $D$ denotes a disk in $\Sigma_h$.

On the other hand, our result is sharp since we have the following example.
\begin{example}
We are able to construct a PSC metric on
$$(\mathbb T^{n-1}-B)\times \mathbb S^1_n\sqcup_\phi \mathbb S^{n-2}\times \mathbb D^2$$
based on the technique in \cite{Miao2002}. Take the product manifold
$$
\left(\mathbb R^{n-1}/\mathbf Z^{n-1}\right)\times \mathbb S^1\left(\frac{1}{4n}\right)
$$
and denote $B$ to be the $(n-1)$-ball in $\mathbb R^{n-1}/\mathbf Z^{n-1}$ with radius $1/4$. Clearly,
$$\left(\left(\mathbb R^{n-1}/\mathbf Z^{n-1}\right)-B\right)\times \mathbb S^1\left(\frac{1}{4n}\right)$$
is a flat manifold whose boundary has the mean curvature $-4(n-1)$ with respect to the unit outer normal. Now we take the product manifold
$$
\mathbb S^{n-1}\left(\frac{1}{4}\right)\times \mathbb D^2\left(\frac{1}{4n}\right).
$$
It has positive scalar curvature and its boundary has mean curvature $4n$ with respect to the unit outer normal. Since the sum of above two mean curvatures is positive, we can glue these two manifolds along the boundary and modify the metric to a smooth one with PSC from \cite{Miao2002}.
\end{example}

The key of our theorem is that the homotopically non-trivial condition on $\phi_2(\gamma)$ leads $(\mathbb T^n,i)\sharp (M_2,\phi_2)$ to be a Schoen-Yau-Schick manifold (see definition in \cite[P 662]{Gromov18}). Namely, we use

\begin{theorem}[Schoen-Yau \cite{SY2017}]\label{Thm: SYS}
Assume that $M$ is a compact oriented $n$-manifold with
a metric of positive scalar curvature. If $\alpha_1,\ldots,\alpha_{n-2}$ are classes in
$H^1(M, \mathbf Z)$ with the property that the class $\sigma_2$ given by {$$\sigma_2 = [M]\frown(\alpha_1\smile\cdots\smile \alpha_{n-2})\neq 0\in H_2(M, \mathbf Z),$$ }
then the class $\sigma_2$ can be
represented by a sum of smooth two spheres. If $\alpha_{n-1}$ is any class in
$H^1(M, \mathbf Z)$, then we must have {$\alpha_{n-1} \smile \sigma_2 = 0$}. In particular, if $M$ has
classes $\alpha_1,\ldots, \alpha_{n-1}$ with
{$$
[M]\frown(\alpha_1\smile\cdots\smile \alpha_{n-1})\neq 0\in H_1(M,\mathbf Z)
$$} then $M$ cannot carry a metric of positive scalar curvature.
\end{theorem}

However, we weaken the homologically non-trivial condition to the homotopically non-trivial one (see examples in \eqref{Eq: homotopy non-trivial example} for essential difference). This turns out to be crucial for our further discussions on the fill-in problem for $2$-torus.

\begin{proof}[Proof for Theorem \ref{Thm: homotopy nontrivial no PSC}]
Denote $\mathbb S^1_i$ to be the $i$-th circle component of $\mathbb T^{n-1}$. Imagine $\mathbb T^{n-1}$ as the quotient space of $n$-cube $I^{n-1}=[-1,1]^{n-1}$ from identifying opposite faces and take $B$ as a small ball centered at the origin. In the following, we use
$$\mathbb S^1_{i_1}\times \mathbb S^1_{i_2}\times\cdots\times \mathbb S^1_{i_k},\quad 1\leq i_1< i_2<\cdots i_k\leq n-1,\quad k\leq n-2,$$
to denote the $k$-torus from the cells of $\partial I^{n-1}$. Clearly, these $k$-toruses can be considered as submanifolds in $(\mathbb T^n,i)\sharp (M_2,\phi_2)$. For convenience, we will denote $M=(\mathbb T^n,i)\sharp (M_2,\phi_2)$ in the following. Since the intersection number of
$$
\mathbb S^1_{i_1}\times \mathbb S^1_{i_2}\times\cdots\times \mathbb S^1_{i_k}\times \mathbb S^1_n\quad\text{and}\quad \mathbb S^1_{j_1}\times \mathbb S^1_{j_2}\times\cdots\times \mathbb S^1_{j_l}
$$
is non-zero when $(i_1,\ldots,i_k,n,j_1,\ldots,j_l)$ is a permutation of $(1,2,\ldots,n)$, the $(k+1)$-torus
$$\mathbb S^1_{i_1}\times \mathbb S^1_{i_2}\times\cdots\times \mathbb S^1_{i_k}\times \mathbb S^1_n$$
represents a non-trivial homology class in $H_{k+1}(M,\mathbf Z)$. In particular, we can take the Poincar\'e dual $\alpha_i$ of $[\mathbb S^1_1\times\cdots\times\hat{\mathbb S}^1_i\times\cdots\times \mathbb S^1_n]$ in $H^1(M,\mathbf Z)$ for $i$ from $1$ to $n-2$, where we use $\hat S^1_i$ to indicate that $\mathbb S^1_i$ is removed. Notice also that
{\begin{equation}\label{Eq: non zero h2 class}
[M]\frown(\alpha_1\smile\alpha_2\smile\cdots\smile\alpha_{n-2})=[\mathbb S^1_{n-1}\times \mathbb S^1_n]\neq 0\in H_2(M,\mathbf Z).
\end{equation}
}
To prove the first part of the theorem, we work on $(M,g)$ with a PSC metric and try to deduce a contradiction. It follows from Theorem \ref{Thm: SYS} that the homology class $[\mathbb S^1_{n-1}\times \mathbb S^1_n]$ can be represented by a union of smooth spheres, say $\Sigma$.
Now we consider the intersection of $\Sigma$ and $S=\mathbb S^1_1\times\cdots\times{\mathbb S}^1_{n-2}\times \mathbb S^1_n$. Without loss of generality, we can assume that they intersect transversely. It is well-known that the intersection $\Sigma\cap S$ is homologous to $\mathbb S^1_n$ in $H_1(S,\mathbf Z)$ (see Lemma 4.2 and relation (3) in \cite{hutchings}). It is clear that $\Sigma\cap S$ is represented by finitely many closed curve
$$
c_i:[0,1]\to S,\quad c_i(0)=c_i(1),\quad i=1,2,\ldots,N.
$$
Fix a point $x_0$ in $\mathbb S^1_n$. From the connectedness of $S$, we can take a path $P_i:[0,1]\to S$ for each $i$ such that $P_i(0)=x_0$ and $P_i(1)=c_i(0)$. Define
$$
c=(P_1c_1P_1^{-1})(P_2c_2P_2^{-1})\cdots(P_Nc_NP_N^{-1}).
$$
Then $c$ represents an element in $\pi_1(S,x_0)$ such that $i_*([c])=[\mathbb S^1_n]$, where $i_*$ denotes the Hurewicz homomorphism
$$i_*:\pi_1(S,x_0)\to H_1(S,\mathbf Z).$$
Notice that $S$ is an $(n-1)$-torus and so $i_*$ is an isomorphism, which implies that $c$ is homotopic to $\mathbb S^1_n$. Since each component of $\Sigma_{n-2}$ is a sphere, each $c_i$ represents a trivial free homotopy class. As a result, the curve $c$ is homotopic to a point in $M$ and so is $\mathbb S^1_n$. Notice that $\mathbb S^1_n$ is homotopic to $\gamma$. However, it comes from { Proposition \ref{Prop: same homotopy local and global}} that $\gamma$ is homotopically non-trivial in $M$, which leads to a contradiction.

The second part of the theorem comes easily from the first one. Given an arbitrary smooth metric $g$ on $(\mathbb T^n,i)\sharp (M_2,\phi_2)$ with non-negative scalar curvature, the non-existence of any PSC metric combined with Ricci flow (or a similar deformation) yields the Ricci-flatness of the metric $g$. Notice that the first Betti number of $M$ is no less than $n-2$ from \eqref{Eq: non zero h2 class}. It follows from \cite[Theorem 3]{CG1971} that the universal covering of $(M,g)$ is the Euclidean space $\mathbb E^n$ and we complete the proof.
\end{proof}

Next we take $i:\Omega\to\mathbb T^n$ to be a tube neighborhood of $\mathbb \mathbb S^1_{n-1}\times \mathbb S^1_n$. We show the following result:
\begin{theorem}\label{2torus}
Denote $\Sigma$ to be a $2$-torus in $\partial \Omega$ such that $\Sigma$ is homotopic to $\mathbb S^1_{n-1}\times\mathbb S^1_n$ in $\mathbb T^n-\Omega$. Suppose that the inclusion map
$i_*:\pi_1(\phi_2(\Sigma))\to \pi_1(M_2-\phi_2(\Omega))$
 is injective for some $(M_2,\phi_2)$ in $\mathcal C_\Omega$.
Then $(\mathbb T^n,i)\sharp (M_2,\phi_2)$ admits no PSC metric for $n\leq 7$. Moreover, if $g$ is a smooth metric on $(\mathbb T^n,i)\sharp (M_2,\phi_2)$ with non-negative scalar curvature, then it is flat.
\end{theorem}

\begin{proof}
As before, we take $\alpha_i$ to be the Poincar\'e dual of the homology class $[\mathbb S^1_1\times\cdots\times\hat{\mathbb S}^1_i\times\cdots\times \mathbb S^1_n]$ in $H^1(M,\mathbf Z)$ for $i$ from $1$ to $n-3$. Clearly,
$$
[M]\frown(\alpha_1\smile\cdots\smile\alpha_{n-3})=[\mathbb S^1_{n-2}\times \mathbb S^1_{n-1}\times \mathbb S^1_n]\neq 0\in H_3(M,\mathbf Z).
$$
Based on the dimension reduction argument behind Theorem \ref{Thm: SYS}, we can find an orientable embeded submanifold $\Sigma_3$ with interger multiplicity homologous to $[\mathbb S^1_{n-2}\times \mathbb S^1_{n-1}\times \mathbb S^1_n]$, which admits a PSC metric. As a result, $\Sigma_3$ has a non-empty intersection with $S=\mathbb S^1_1\times\cdots\times{\mathbb S}^1_{n-3}\times \mathbb S^1_{n-1}\times \mathbb S^1_n$. Without loss of generality, we can assume that they intersects transversely. So the intersection is an orientable embedded surface with integer multiplicity, denoted by $\Sigma_2$. Let us work on any connected component $\Sigma_2'$ of $\Sigma_2$. Of course, it is contained in some connected component $\Sigma_3'$ of $\Sigma_3$. We are going to show that $\Sigma_2'$ can be expressed as the sum of some $2$-spheres as cycles in $S$. If $\Sigma_2'$ itself is a $2$-sphere, then no more work need to be done. Otherwise, we minimize the area functional in the isotopy class of $\Sigma_2'$ in $\Sigma_3'$ with respect to a chosen PSC metric. There are two cases:
\begin{itemize}
\item[(i)] The area of surfaces in the isotopy class of $\Sigma_2'$ has a uniform positive lower bound. In this case, we take a minimizing sequence $\Sigma_{2,k}'$ in $\Sigma_3'$. From \cite{MSY1982} we have
$$
\Sigma_{2,k}'\to n_1S_1+n_2S_2+\cdots+n_pS_p
$$
in the sense of varifold, where each $S_j$ is a smooth embedded minimal surface satisfies
$$
\int_{S_j}(|A|^2+\Ric(\nu,\nu))\phi^2\,\mathrm d\sigma\leq \int_{S_j}|\nabla\phi|^2\,\mathrm d\sigma.
$$
Here $\nu$ is a unit normal vector field (allowed to be discontinuous) of $S_j$ in $\Sigma_3'$. Since $\Sigma_3'$ has PSC, the standard Schoen-Yau's trick yields that $S_j$ is a projective plane or a $2$-sphere. It follows from Remark 2 on page 625 and  Remark 3.27 on page 635 in \cite{MSY1982} that $\Sigma_2'$ is diffeomorphic to the ``connected sum'' of several spheres in $\Sigma_3'$. Here we abuse the concept of connected sum to mean that spheres are connected by several thin necks. In particular, there are disjoint simple closed curves $\gamma_1',\ldots,\gamma_N'$ on $\Sigma_2'$ such that each $\gamma_s'$ is homotopic to a point in $\Sigma_3'$ and the complement
$\Sigma_2'-\cup_{s=1}^N\gamma_s'$ is a union of punctured $2$-spheres. The situation is illustrated by following figure.
\begin{figure}[htbp]
\centering
\includegraphics[width=7cm]{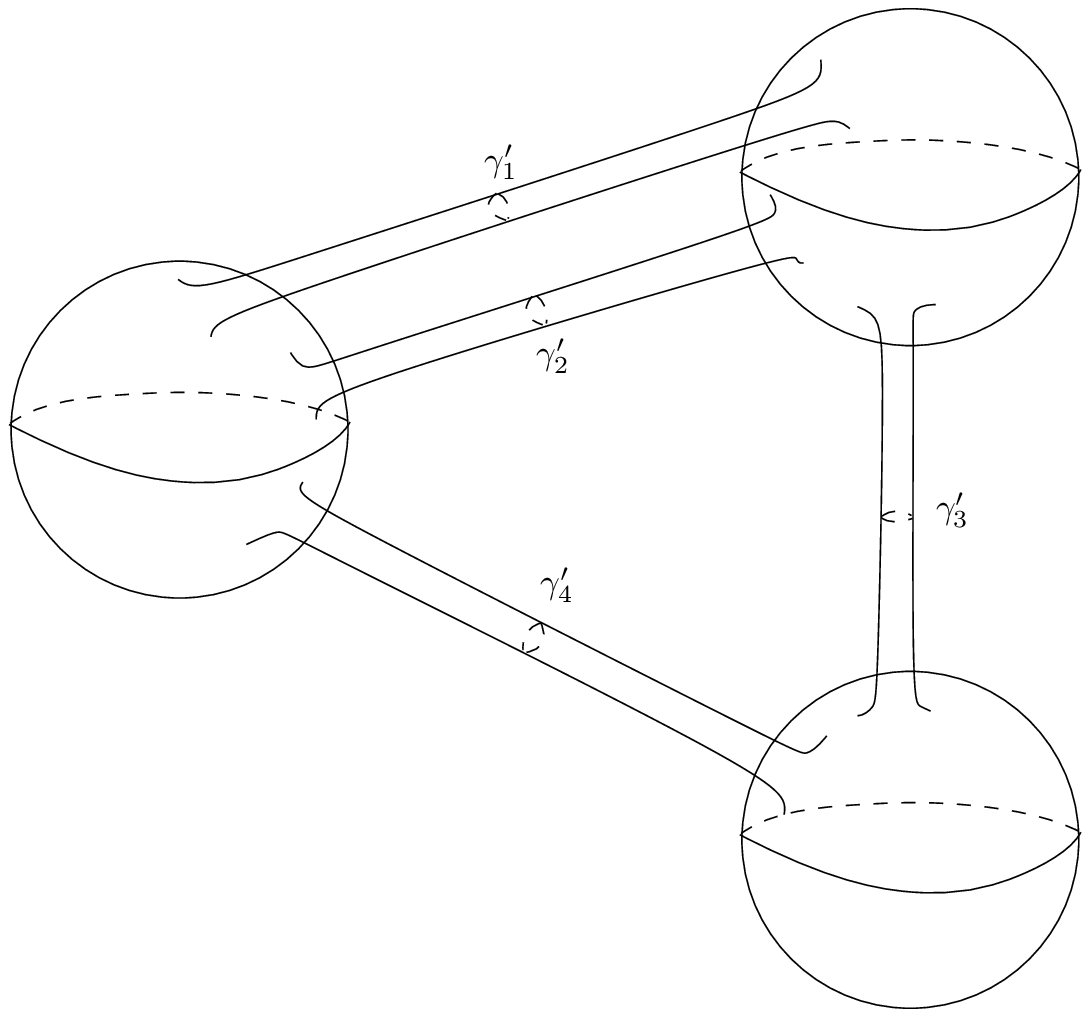}
\end{figure}
\item[(ii)] The area of surfaces in the isotopy class of $\Sigma_2'$ can be arbitrarily small. From the homological systole estimate for closed surfaces (see \cite[P71]{Pu1952} and \cite[P301]{Gromov1996}), we can deduce the following consequence: if an orientable closed surface $\Sigma_h$ with genus $h$ has area $A_0$, then we can find $h$ disjoint simple closed curves $\gamma_1,\ldots,\gamma_h$ on $\Sigma_h$ with length no greater than $C(h)A_0^{1/2}$ such that $\Sigma_h-\cup_{s=1}^h\gamma_h$ is a punctured $2$-sphere. In particular, a closed surface $\Sigma_2''$ can be chosen in the isotopy class of $\Sigma_2'$ such that the corresponding curves $\gamma_1'',\ldots,\gamma_h''$ have lengths no greater than the convex radius of $\Sigma_3'$. This implies that $\gamma_s''$ is homotopically trivial in $\Sigma_3'$ and so there are several simple closed curves $\gamma_1',\ldots,\gamma_h'$ on $\Sigma_2'$ homotopic to a point in $\Sigma_3'$ such that $\Sigma_2'-\cup_{s=1}^h\gamma_s'$ is a punctured $2$-sphere. The situation is shown as following
\begin{figure}[htbp]
\centering
\includegraphics[width=11cm]{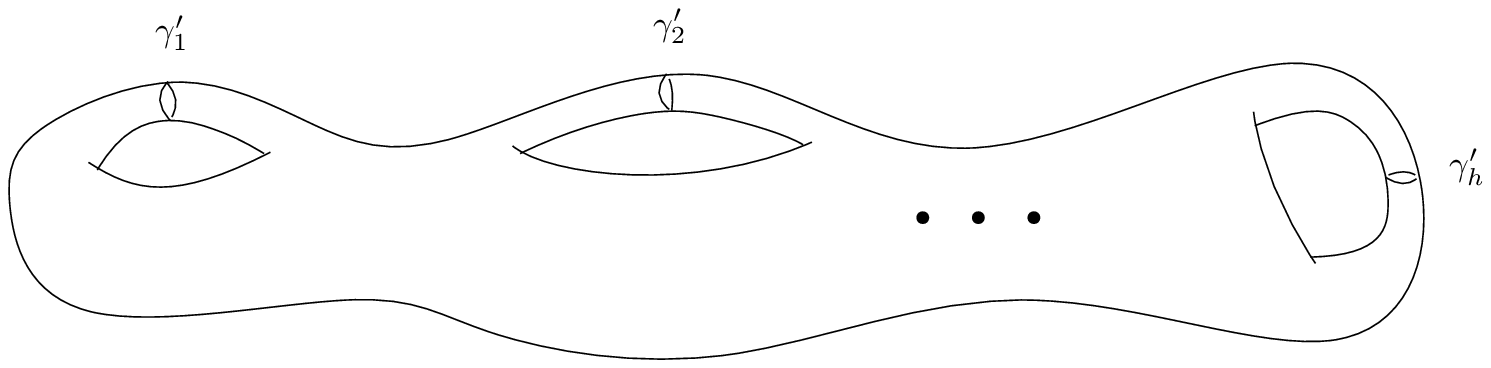}
\end{figure}
\end{itemize}
Now we show that each closed curve $\gamma_s'$ is also homotopic to a point in $S$. Do decomposition
$$
[\gamma_s']=c_1[\mathbb S^1_1]+\cdots c_{n-3}\mathbb [S^1_{n-3}]+c_{n-1}[\mathbb S^1_{n-1}]+c_n[\mathbb S^1_n]\quad\text{in}\quad H_1(S,\mathbf Z).
$$
Since $\gamma_s'$ is homotopic to a point in $\Sigma_3'$, it is homologous to zero in $M$. So
\begin{equation*}
c_1= [\gamma_s']\cdot [\mathbb S^1_2\times\cdot\times\mathbb S^1_n]=0,
\end{equation*}
and the same thing holds for $c_2,\ldots,c_{n-3}$. From Proposition \ref{Prop: same homotopy local and global} it follows that the $2$-torus $\mathbb S^1_{n-1}\times \mathbb S^1_n$ is incompressible in $M$. Since the Hurewicz homomorphism gives an isomorphism between $H_1(S,\mathbf Z)$ and $\pi_1(S)$, we have $c_{n-1}=c_n=0$ and so $\gamma'_s$ is homotopic to a point in $S$. By adding caps with inverse orientations in $S$, $\Sigma_2'$ can be written as a sum of spherical cycles in $S$. From $\pi_2(S)=0$ we know that $\Sigma_2$ is homologous to zero in $S$. On the other hand, $\Sigma_2$ is homologous to $\mathbb S^1_{n-1}\times \mathbb S^1_n$ in $S$, which is homologically non-trivial in $S$. So we obtain a contradiction.
The rigidity part follows from a similar argument as in the proof of Theorem \ref{Thm: homotopy nontrivial no PSC}.
\end{proof}

\section{Positive mass theorem for ALF\&ALG manifolds}\label{Sec: PMT}
Next, we give an application of Theorem \ref{Thm: homotopy nontrivial no PSC} and Theorem \ref{2torus}, i.e. to show the positive mass theorem of special ALF\&ALG manifolds. The ALF\&ALG we mean in this paper are defined blew.

\begin{definition}\label{ALF}
Let $(M^n,g)$ be a complete and noncompact Riemannian manifold, and it is ALF with asymptotic order $\mu$ if it satisfies:
\begin{itemize}
	\item There is a compact set $K\subset M$ so that $M\setminus K$ is diffeomorphic to $(\mathbb R^{n-1}\setminus \mathbb B^{n-1}(r) )\times \mathbb S^1$, $\mathbb B^{n-1}(r)$ is a ball in $\mathbb{R}^{n-1}$ with radius $r$;
	\item  when $n=3$, $g=dr^2 +\beta^2 r^2 d\phi^2 +l^{2}d\theta^2 + \sigma$ on $M\setminus K$, $\beta\in\mathbb{R}^{+}$, $l\in\mathbb{R}^{+}$, $(r, \phi)$ is the  polar coordinates on $\mathbb R^2$, $\theta$ is the standard coordinate on $\mathbb S^1$;
    \item  when $n\geqslant 4$, $g=(1+\frac{m}{2r^{n-3}})^{\frac{4}{n-3}}dx^{2} +l^{2}d\theta^2 + \sigma$ on $M\setminus K$, $m\in\mathbb{R}$, $l\in\mathbb{R}^{+}$, $x=\{x^{i}\}_{i=1}^{n-1}$ is the standard cartesian coordinates on $\mathbb{R}^{n-1}$, $\theta$ is the standard coordinate on $\mathbb S^1$;
	\item $\sigma$ is the error term, satisfy $\sum\limits^2_{k=0}\sum\limits_{|\alpha|=k}r^k|\partial^\alpha \sigma| =O(r^{-\mu})$, $\mu>n-3$, as $r\rightarrow \infty$, $r=|x|$, $|\cdot|$ is the Euclidean norm, $\partial$ is $\partial_{i}$ or $\partial_{\theta}$, $i=1,\cdots ,n-1$.
	\end{itemize}

And $\mathbb{S}^1$  called the $\mathbb S^1$-factor at the infinity of $M$.
\end{definition}

\begin{definition}\label{ALG}
Let $(M^n,g)$ be a complete and noncompact Riemannian manifold, and it is ALG with asymptotic order $\mu$ if it satisfies:
\begin{itemize}
	\item There is a compact set $K\subset M$ so that $M\setminus K$ is diffeomorphic to $(\mathbb R^{n-2}\setminus \mathbb B^{n-2}(r) )\times \mathbb T^2$, $\mathbb B^{n-2}(r)$ is a ball in $\mathbb{R}^{n-2}$ with radius $r$;
	\item  when $n=4$, $g=dr^2 +\beta^2 r^2 d\phi^2 +l^{2}d\gamma^2 +\sigma$ on $M\setminus K$, $\beta\in\mathbb{R}^{+}$, $l\in\mathbb{R}^{+}$, $(r, \phi)$ is the  polar coordinates on $\mathbb R^2$, $d\gamma^2$ is flat metric on $\mathbb{T}^{2}$ with area equals to $4\pi^{2}$;
    \item  when $n>4$, $g=(1+\frac{m}{2r^{n-4}})^{\frac{4}{n-4}}dx^{2} +l^{2}d\gamma^2 +\sigma$ on $M\setminus K$, $m\in\mathbb{R}$, $l\in\mathbb{R}^{+}$, $x=\{x^{i}\}_{i=1}^{n-1}$ is the standard cartesian coordinates on $\mathbb{R}^{n-2}$, $d\gamma^2$ is a flat metric on $\mathbb{T}^{2}$ with area equals to $4\pi^{2}$;
	\item $\sigma$ is the error term, satisfy $\sum\limits^2_{k=0}\sum\limits_{|\alpha|=k}r^k|\partial^\alpha \sigma|=O(r^{-\mu})$, $\mu>n-4$, as $r\rightarrow \infty$, $r=|x|$, $|\cdot|$ is the Euclidean norm, $\partial$ is $\partial_{i}$ or $\partial_{\gamma^{j}}$, $i=1,\cdots, n-2$, $j=1,2$, $\{\gamma^{j}\}_{j=1}^{2}$ is the local coordinates of $\mathbb{T}^{2}$.
	\end{itemize}
And $\mathbb{T}^2$ called the $\mathbb T^2$-factor at the infinity of $M$.

\end{definition}

We need to assume the order of $\sigma$ in our theorems, for convenience,we define
$\mu$-\emph{condition} as following: if $(M^n,g)$ is ALF ,then $ \mu>1 $ when $n=3$, $\mu>2n-6$ when $n\geqslant 4$; if $(M^n,g)$ is ALG, then $\mu > 1$ when $n=4$, $\mu>2n-8$ when $n\geqslant 5$.

Now we state our main Theorems:

\begin{theorem}\label{ALFPMT}
Let $(M^n,g)$ be an ALF manifold with scalar curvature $R_g\geqslant 0$ and $\mu$-condition. If $n\leqslant 7 $ and its $\mathbb S^1$-factor at the infinity represents a non-trivial element in $\pi_1(M^{n})$, then

\begin{center}
$\beta\leqslant 1$, if $n=3$ ;  $m\geqslant 0$, if $n\geqslant 4$
\end{center}

moreover, if we assume $\sum\limits^4_{k=3}\sum\limits_{|\alpha|=k}r^k|\partial^\alpha \sigma|=O(r^{-\mu})$, then if $\beta=1$, $M^{3}$ is isometric to flat $\mathbb{R}^{2}\times \mathbb{S}^{1}$, with $\mathbb{S}^{1}$ length $2\pi l$; if $m=0$, $M^{n}$ is isometric to flat $\mathbb{R}^{n-1}\times \mathbb{S}^{1}$, with $\mathbb{S}^{1}$ length $2\pi l$.

\end{theorem}

\begin{theorem}\label{ALGPMT}
Let $(M^n,g)$ be an ALG manifold with scalar curvature $R_g\geqslant 0$ and $\mu$-condition. If $n\leqslant 7 $ and $i:\mathbb{T}^{2}\rightarrow M^{n}$ be the inclusion, its induced  map $i_{\ast}:\pi_{1}(\mathbb{T}^{2})\rightarrow \pi_{1}(M^{n})$ is injective, then

\begin{center}
$\beta\leqslant 1$, if $n=4$ ;  $m\geqslant 0$, if $n\geqslant 5$
\end{center}

moreover, if we assume $\sum\limits^4_{k=3}\sum\limits_{|\alpha|=k}r^k|\partial^\alpha \sigma|=O(r^{-\mu})$,  then if $\beta=1$, $M^{4}$ is isometric to flat $\mathbb{R}^{2}\times \mathbb{T}^{2}$, with $\mathbb{T}^{2}$ area $4\pi^{2} l^{2}$; if $m=0$, $M^{n}$ is isometric to flat $\mathbb{R}^{n-2}\times \mathbb{T}^{2}$, with $\mathbb{T}^{2}$ area $4\pi^{2} l^{2}$.

\end{theorem}

We mainly work on the proof of Theorem \ref{ALFPMT}, and divide our proof into inequality part and rigidity part. For ALG situation, proof of inequality part is totally same as that of ALF situation, and the analysis of   Minerbe's work \cite{minerbe2008a} and appendix in \cite{MR2855540} help us  get the rigidity part of Theorem \ref{ALGPMT}

Without loss of generality, we assume $l=1$ for ALF\&ALG manifolds in this section.

\subsection{Proof of inequality part}

We consider an ALF manifold $(M^{n},g)$  with scalar curvature $R_g\geqslant 0$ and $\mu$-condition. In $n=3$ and $n\geqslant 4$, some analysis is different, but idea is same, i.e. if $\beta>1$ or $m<0$, we can construct a new metric $\tilde{g}$ that connects $g$ with flat metric outside a large coordinate ball and keep nonnegative scalar curvature . Smooth $\tilde{g}$ and paste oppsite faces of large coordinate cube, then apply compact result Theorem \ref{Thm: homotopy nontrivial no PSC} and Theorem \ref{2torus}. Now we assume $\beta>1$ when $n=3$, and $m<0$ when $n\geqslant 4$.

Let $\Omega_{r}=(\mathbb{B}^{n-1}(r)\times \mathbb{S}^{1} \cap M^{n})\cup K$, $\partial\Omega_{r}= \mathbb{S}^{n-2}(r)\times \mathbb{S}^{1}$, for $r$ enough large, and define $\sigma_{r}$ as the restriction of $\sigma$ on $\partial\Omega_{r}$. We always assume $r_{0}$ sufficiently large, and $\lambda \leqslant 1$. Now we define a metric $\tilde{g}$  on $M^{n}$:

For $n=3$:

\begin{equation}\label{PSCExtensionmetric1}
\tilde{g}=\left\{
\begin{aligned}
g,& \quad \text{if $r\leqslant r_0$},\\
u^2dr^2+\beta^2
r^2 d\phi^2 +d\theta^2+[1-\frac{1}{\lambda}(r-r_{0})]\sigma_{r}, &\quad \text{if $r_0\leqslant r\leqslant  r_0+\lambda$},\\
\beta^{2}dx^{2}+d\theta^2, & \quad \text{if $r_0+\lambda\leqslant r$}
\end{aligned}
\right.
\end{equation}

For $n\geqslant 4$:

\begin{equation}\label{PSCExtensionmetric2}
\tilde{g}=\left\{
\begin{aligned}
g,& \quad \text{if $r\leqslant r_0$},\\
u^2dr^2+(1+\frac{m}{2r^{n-3}})^{\frac{4}{n-3}}
r^2 d\phi^2 +d\theta^2+[1-\frac{1}{\lambda}(r-r_{0})]\sigma_{r}, &\quad \text{if $r_0\leqslant r\leqslant  r_0+\lambda$},\\
(1+\frac{m}{2(r_{0}+\lambda)^{n-3}})^{\frac{4}{n-3}}dx^{2}+d\theta^2, & \quad \text{if $r_0+\lambda\leqslant r$}
\end{aligned}
\right.
\end{equation}

where $dx^{2}=dr^2+r^2 d\phi^2$, $d\phi^2$ is the standard metric on $\mathbb{S}^{n-2}$, $u$ is a smooth and positive  function on $\Omega_{r_{0}+\lambda}\setminus\Omega_{r_{0}}$. Note that $\tilde g$ is smooth on $M\setminus (\partial\Omega_{r_{0}} \cup \mathbb \partial\Omega_{r_{0}+\lambda})$ and is Lipschitz on $M$.

For convenience, we define

\begin{equation}
\gamma_r=\left\{\begin{aligned}
\beta^2 r^2 d\phi^2 +d\theta^2+[1-\frac{1}{\lambda}(r-r_{0})]\sigma_{r},n=3 \\
(1+\frac{m}{2r^{n-3}})^{\frac{4}{n-3}}r^2 d\phi^2 +d\theta^2+[1-\frac{1}{\lambda}(r-r_{0})]\sigma_{r},n\geqslant 4
\end{aligned}\right.
\end{equation}

Note that $\gamma_r$ defined on $\partial\Omega_{r}$, $r_{0}\leqslant r\leqslant r_{0}+\lambda$,and we define $\bar g=dr^2+\gamma_r $ on $\Omega_{r_{0}+\lambda}\setminus\Omega_{r_{0}}$. $H_{g}(r)$ defined the mean curvature of $g$ in $\partial\Omega_{r}$ respected to outward normal vector, similarly for $H_{\bar{g}}(r)$. And when $\tilde{g}$ is not smooth on $\partial\Omega_{r}$,   $H_{\tilde{g}}^{+}(r)$ defined the mean curvature of $\tilde{g}$ respected to interior metric, $H_{\tilde{g}}^{-}(r)$ respected to exterior metric.

\begin{proposition}\label{compactfy1}
Let $(M^n,g)$ be an ALF manifold with $R_{g}\geqslant 0$ and $\mu$-condition.

If $n=3$, and $\beta>1$, then for $r_{0}$ large enough, and $\lambda=1$, there is a smooth and positive  function $u$  on $\Omega_{r_{0}+\lambda}\setminus\Omega_{r_{0}}$, so that $R_{\tilde{g}}$	is  nonnegative on $M\setminus (\partial\Omega_{r_{0}} \cup \mathbb \partial\Omega_{r_{0}+\lambda})$, and
$$H_{\tilde{g}}^{+}(r_{0})=H_{\tilde{g}}^{-}(r_{0}),H_{\tilde{g}}^{+}(r_{0}+\lambda)>H_{\tilde{g}}^{-}(r_{0}+\lambda)$$

If $n\geqslant 4$, and $m<0$, then for $r_{0}$ large enough, and $\lambda=r_{0}^{-\epsilon}$, $n-4<\epsilon<\mu-(n-2)$, there is a smooth and positive  function $u$  on $\Omega_{r_{0}+\lambda}\setminus\Omega_{r_{0}}$, so that $R_{\tilde{g}}$	is  nonnegative on $M\setminus (\partial\Omega_{r_{0}} \cup \mathbb \partial\Omega_{r_{0}+\lambda})$, and
$$H_{\tilde{g}}^{+}(r_{0})=H_{\tilde{g}}^{-}(r_{0}), H_{\tilde{g}}^{+}(r_{0}+\lambda)>H_{\tilde{g}}^{-}(r_{0}+\lambda)$$

 \end{proposition}

\begin{proof}
By the proof of Lemma 2.1 in \cite{shi2020total}, the  requirement that $R_{\tilde{g}} =0 $ on $\Omega_{r_{0}+\lambda}\setminus\Omega_{r_{0}}$, and $H_{\tilde{g}}^{+}(r_{0})=H_{\tilde{g}}^{-}(r_{0})$ is equivalent to the following initial value problem.

\begin{equation}\label{PSCExtensionQSeq}
\left\{
\begin{aligned}
H_{\bar{g}}(r)\frac{\partial u}{\partial r}&=u^2\Delta_{\gamma_r}u-\frac{1}{2}R_{\gamma_r}u^3+\frac{1}{2}\left(R_{\gamma_r}-R_{\bar g}\right)u, r_{0}\leqslant r\leqslant r_{0}+\lambda\\
\quad u(r_0,\cdot)&=\frac{H_{\bar{g}}(r_{0})
}{H_{g}(r_{0})}
 \end{aligned}
\right.
\end{equation}

By a direct computation, we see that for $r_{0}\leqslant r\leqslant r_{0}+\lambda $ ,

\begin{equation}\label{meancurvature1}
\begin{split}
H_{g}(r)&=\frac{1}{r}+O(r^{-(\mu+1})	, n=3\\
H_{g}(r)&=\frac{n-2}{r}-\frac{(n-2)^{2}m}{(n-3)}\frac{1}{r^{n-2}}+O(r^{-(n-1)})+ O(r^{-(\mu+1)}), n\geqslant 4
\end{split}
\end{equation}

From the hypothesis of order $\mu$, the term $O(\frac{1}{r^{\mu+1}})$ thrown .
In $n\geqslant 4$, $\mu> 2n-6$, we always assume $\mu \leqslant 2n-5$, if not, just let $\mu = 2n-5$, same, in $n=3$, we assume $1<\mu\leqslant2$. Then we calculate $H_{\bar{g}}(r)$, note that the term $[1-\frac{1}{\lambda}(r-r_{0})]\sigma_{r} = \eta_{r}$, with $\eta_{r}=O(r^{-\mu})$, $\partial \eta_{r}=\frac{1}{\lambda}O(r^{-\mu})$, $\partial^{2}\eta_{r}=\frac{1}{\lambda}O(r^{-\mu-1})$ on $\Omega_{r_{0}+\lambda}\setminus\Omega_{r_{0}}$.

For $n=3 $, take $\lambda =1$, then
\begin{equation}\label{bargn=3}
H_{\bar{g}}(r)=\frac{1}{r}+O(r^{-\mu})>0
\end{equation}

For $n\geqslant 4$, let $\lambda =r_{0}^{-\epsilon}$, $n-4<\epsilon<\mu-(n-2)$, then
\begin{equation}\label{bargn=4}
H_{\bar{g}}(r)=\frac{n-2}{r}-\frac{(n-2)m}{r^{n-2}}+O(r^{-(\mu-\epsilon)})
\end{equation}

Note that
$$u(r_0,\cdot)=\frac{H_{\bar{g}}^{-}(r_{0})}{H_{\tilde{g}}^{+}(r_{0})}$$

so
\begin{equation}\label{ur0}
\begin{aligned}
u(r_0,\cdot)&=1 +O(r^{1-\mu}_0),n=3\\
u(r_0,\cdot)&=1+\frac{1}{(n-3)}\frac{m}{r_{0}^{n-3}}+O(r_{0}^{-(\mu-\epsilon-1)}), n\geqslant 4
\end{aligned}
\end{equation}

here we have use the relationship $n-4<\epsilon<\mu-(n-2)$. We want to show the existence of $u$ in \eqref{PSCExtensionQSeq} and control $u(r_{0}+\lambda,\cdot)$. As in the proof of Lemma 2.1 in \cite{shi2020total},  we set

\begin{equation}\label{M}
M=\max\limits_{\Omega_{r_{0}+\lambda}\setminus\Omega_{r_{0}}} H_{\bar{g}}^{-1}(r)\left(|R_{\gamma_r}|+|R_{\bar g}|\right)
\end{equation}

Here $R_{\gamma_r}$ means the scalar curvature of $\partial\Omega_{r}$ with metric $\gamma_{r}$, $R_{\bar g}$ means the scalar curvature of $\bar{g}$ on $\Omega_{r_{0}+\lambda}\setminus\Omega_{r_{0}}$. Then by a direct computation, 
$R_{\gamma_r}=O(r^{-(\mu+2)})$, $R_{\bar g}=O(r^{-(\mu+1)})$, in $n=3$; $R_{\gamma_r}=\frac{1}{r^{2}}R_{\mathbb{S}^{n-2}}+O(r^{-(\mu+2)})$, in $n\geqslant 4$, here $R_{\mathbb{S}^{n-2}}$ means the scalar curvature of the standard unit sphere with dimension $n-2$,  and $R_{\bar g}=O(r^{-(n-1)})$. See that

\begin{equation}\label{Mestimate}
M\leqslant C r_{0}^{-1}
\end{equation}

here $C>0$ is a positive constant depends only on the metric $g$ and $n$, and may change from line to line.
Then by the proof of CLAIM 2.2 in \cite{shi2020total}, we know that \eqref{PSCExtensionQSeq} has a positive solution on  $\Omega_{r_{0}+\lambda}\setminus\Omega_{r_{0}}$, and

\begin{equation}\label{u(r0+lambda)}
u(r_{0}+\lambda,\cdot)\leqslant \frac{u(r_{0})}{\sqrt{(1+u(r_{0})^{2})e^{-M\lambda}-u(r_{0})^{2}}}\leqslant u(r_{0})(1+\frac{C\lambda}{r_{0}})
\end{equation}

For $n=3$, $\lambda =1$, so

$$
u( r_0+\lambda,\cdot)\leqslant 1+O(r_{0}^{1-\mu})
$$

For $n\geqslant 4$, $\lambda =r_{0}^{-\epsilon}$, so

$$
u( r_0+\lambda,\cdot)\leqslant 1+\frac{1}{(n-3)}\frac{m}{r_{0}^{n-3}}+o(r_{0}^{-(n-3)})
$$

For $n=3$, the mean curvature of $\tilde{g}$

\begin{equation}\label{meancurvature2}
 H_{\tilde{g}}^{+}(r_{0}+\lambda)=\frac{H_{\bar{g}}(r_{0}+\lambda)}{u(r_{0}+\lambda,\cdot)}\geqslant (r_0+\lambda)^{-1}+O((r_0+\lambda)^{-\mu}),
\end{equation}

meanwhile, we have

\begin{equation}\label{meancurvature3}
H_{\tilde{g}}^{-}(r_0+\lambda)=\beta^{-1} (r_0+\lambda)^{-1},
\end{equation}

hence
$$
H_{\tilde{g}}^{+}(r_{0}+\lambda)> H_{\tilde{g}}^{-}(r_0+\lambda)
$$

provided by $\beta >1$ and $r_0$ large is enough.

For $n\geqslant 4$

\begin{equation}\label{meancurvature4}
 H_{\tilde{g}}^{+}(r_{0}+\lambda)=\frac{H_{\bar{g}}(r_{0}+\lambda)}{u(r_{0}+\lambda,\cdot)}\geqslant
\frac{n-1}{r_{0}+\lambda}-\frac{(n-2)^{2}}{n-3}\frac{m}{(r_{0}+\lambda)^{n-2}}+o((r_{0}+\lambda)^{-(n-2)})
\end{equation}

and

\begin{equation}\label{meancurvature5}
H_{\tilde{g}}^{-}(r_0+\lambda)=\frac{n-2}{r_{0}+\lambda}-\frac{n-2}{n-3}\frac{m}{(r_{0}+\lambda)^{n-2}}+o((r_{0}+\lambda)^{-(n-2)})
\end{equation}

hence
$$
H_{\tilde{g}}^{+}(r_{0}+\lambda)> H_{\tilde{g}}^{-}(r_0+\lambda)
$$

provided by $m<0$ and $r_0$ is enough large.Thus we finish the proof of Proposition \ref{compactfy1}.
\end{proof}
\begin{proof}[Proof of the inequality part of Theorem \ref{ALFPMT}]:

 Suppose Theorem \ref{ALFPMT}  fails, $\beta>1$ or $m<0$. Due to Proposition \ref{compactfy1}, we can construct a $(M^{n},\tilde{g})$ defined by \eqref{PSCExtensionmetric1} and \eqref{PSCExtensionmetric2}, satisfied $R_{\tilde{g}}$	is  nonnegative on $M\setminus (\partial\Omega_{r_{0}} \cup \mathbb \partial\Omega_{r_{0}+\lambda})$, and $H_{\tilde{g}}^{+}(r_{0})=H_{\tilde{g}}^{-}(r_{0})$, $H_{\tilde{g}}^{+}(r_{0}+\lambda)>H_{\tilde{g}}^{-}(r_{0}+\lambda)$. Observe that g is flat outside $\Omega_{r_{0}+\lambda}$, then we glue opposite faces of $M\setminus \Box_{r}^{c}\times \mathbb{S}^{1}$, $r>r_{0}+\lambda$, $\Box_{r}^{c}=\mathbb{R}^{n-1}\setminus \{x=\{x^{i}\},|x^{i}|\leqslant r ,\for i =1,\cdots,n-1 \}$, then get a compact manifold $M_{0}^{n}$ with $\tilde{g}$ restricted on it.

Then after carrying Miao's mollifying procedure (\cite{Miao2002}) for corners, we see that there is a smooth metric on $M_{0}^{n}$ with positive scalar curvature which contradicted to Theorem \ref{Thm: homotopy nontrivial no PSC}, hence we finish the proof.

\end{proof}

\subsection{Proof of rigidity part}

Let $(M^{n},g)$ be an ALF manifold with scalar curvature $R_g\geqslant 0$ and $\mu$-condition, and $\beta=1$ when $n=3$, $m=0$ when $n\geqslant4$. First we want to show $R_{g}\equiv 0$, we argue it by contradiction. We assume
$R_g\geqslant 0$ and $R_g(p)>0$ for some $p\in M$. Note that if $g_{\rho}= \rho^{\frac{4}{n-2}}g$ for some  positive smooth function $\rho$ on $M$,  then
$$
R_{ g_{\rho}}= \rho^{-\frac{n+2}{n-2}}(R_g \rho-\frac{4(n-1)}{n-2}\Delta_g \rho).
$$

Assume $R_g>0$ on $B_\varepsilon(p)$ for some small geodesic ball with radius $\varepsilon$, let $f$ be a nonnegative smooth function on $M$ with $f(p)=\frac{R_g(p)}{2}$ and supports in $B_\varepsilon(p)$. We consider the following the boundary value problem for each large $r$.

\begin{equation}\label{u_r}
\left\{
\begin{aligned}
&\Delta_g u_r-\frac{n-2}{4(n-1)}f u_r=0, \quad \text{in $\Omega_r$}	\\
 &u_r|_{\partial\Omega_{r}}=1
 \end{aligned}
\right.
\end{equation}

Note that $f(p)>0$, $w\equiv 0$ is the only solution to the following homogenous equation

\begin{equation}\label{Ur}
\left\{
\begin{aligned}
&\Delta_g w-\frac{n-2}{4(n-1)}f w=0, \quad \text{in $\Omega_r$}	\\
&w|_{\partial\Omega_{r}}=0
 \end{aligned}
\right.
\end{equation}

Thus we see that \eqref{u_r} has a unique positive solution for any large $r$, $\sup u_{r}=1$ and $u_{r}$ is harmonic outside $B_\varepsilon(p)$. To explore the behavior of the solution of \eqref{u_r}, we need the following lemma:

\begin{lemma}\label{estimateBVP1}
Let  $u_r$ be the positive solution of \eqref{u_r}	, then there exists a constant $ \Lambda_{r_{0}} < 1 $ depends on $r_{0}$ and $g$, that
$$
\sup_{r>2r_{0},x\in\partial\Omega_{r_{0}}}|u_r(x)|\leqslant\Lambda_{r_{0}}
$$
\end{lemma}

\begin{proof}

If not, then exists $u_{r_{i}}$, $r_{i}\rightarrow \infty$, or $r_{i}\rightarrow r_{a}\geqslant 2r_{0}$, and
$$\sup_{\partial\Omega_{r_{0}}}|u_{r_{i}}|=d_{i}\rightarrow 1$$

Let $v_{i}=\frac{u_{r_{i}}}{d_{i}}$, then
$$\sup_{\partial\Omega_{r_{0}}}v_{i}=1 ,|v_{i}| \leqslant \frac{1}{d_{i}}$$

If $r_{i}\rightarrow \infty$, then exists a function $v$, $v_{i}\rightarrow v$ uniformly on compact subsets of $M^{n}$, and
\begin{equation}\label{laplacev}
\Delta v-\frac{n-2}{4(n-1)}f v =0,v\leqslant 1
\end{equation}

See that $\sup\limits_{\partial\Omega_{r_{0}}}v=1$, attains its maximum point in interior, then $v\equiv 1$, contradicted to $f(p)\neq 0$. If $r_{i}\rightarrow r_{a}\geqslant 2r_{0}$, by the same argument, we get a $v$ defined in $\Omega_{2r_{0}-\varepsilon}$,  for all $\varepsilon>0$, $v$ satiefied \eqref{laplacev}, and attains its maximum point in interior, contradicted to $f(p)\neq 0$.

\end{proof}
\begin{lemma}\label{estimateBVP2}
Suppose $n=3$, $g$ is ALF with $\beta =1$, $u_{r_{1}}$ defined by $\eqref{u_r}$. $r_{1}$ is sufficient large, then there exists a constant $\delta>0$ depends only on $g$,  $u_{r_{1}}$ satisfied

\begin {equation}\label{11}
\frac{\partial u_{r_{1}}}{\partial r}|_{r=r_{1}}\geqslant \frac{\delta}{r_{1}\log{r_{1}}}
\end{equation}

If $n\geqslant 4$, $g$ is ALF with $m=0$, $u_{r_{1}}$ defined by $\eqref{u_r}$. $r_{1}$ is sufficient large, $\tau >0$, then there exists a constant $\delta>0$ depends only on $g$ and $\tau$, $u_{r_{1}}$ satisfied

\begin {equation}\label{22}
\frac{\partial u_{r_{1}}}{\partial r}|_{r=r_{1}}\geqslant \frac{\delta}{r_{1}^{n-2+\tau}}
\end{equation}
\end{lemma}

\begin{proof}Consider $2 r_{0}\leqslant r_{1}$, $r_{0}$ is large enough and will be fixed later.

For $n=3$, let
$$v=\frac{-1}{(\alpha-2)^{2}}r^{2-\alpha}+C_{1}\log r +C_{2},r=|x|$$

Defined on $\Omega_{r_{1}}\setminus \Omega_{r_{0}}$ and $\alpha>2$. $C_{1},C_{2}$ to be chosen so that
$$v|_{r=r_{0}}=\Lambda_{r_{0}},v|_{r=r_{1}}=1$$

So $(v-u_{r_{1}})|_{\partial (\Omega_{r_{1}}\setminus \Omega_{r_{0}}) }\geqslant 0$. Now $\beta=1$, $g=dx^{2}+d\theta^{2}+\sigma$, so for a $h(r)$

$$\triangle h(r)=h''(r)+\frac{1}{r}h'(r)+h''(r)O(r^{-\mu})+h'(r)O(r^{-1-\mu})$$

Calculation show
\begin{equation}\label{estimateBVP3}
\triangle v=-r^{-\alpha}+C_{1}O(r^{-2-\mu})+O(r^{-\alpha-\mu})
\end{equation}

Note that,
\begin{equation}
C_{1}=\frac{1-\Lambda_{r_{0}}}{\log r_{1}-\log r_{0}}+\frac{1}{(\alpha-2)^{2}}\frac{r_{1}^{2-\alpha}-r_{0}^{2-\alpha}}{\log r_{1}-\log r_{0}}
\end{equation}

Fix $2<\alpha<2+\mu $, and $r_{0}$ large, so that $r_{0}^{2-\alpha}\ll 1$. Then let $r_{1}$ large enough, there exists a $\delta$,

\begin{center}
$C_{1} \geqslant \frac{\delta}{\log{r_{1}}}$
, and $\triangle v <0$ on $\Omega_{r_{1}}\setminus \Omega_{r_{0}}$
\end{center}

By the maximum principle, we have

\begin{center}
$v-u_{r_{1}}\geqslant 0$ on $\Omega_{r_{1}}\setminus \Omega_{r_{0}}$
\end{center}

Due to $v-u_{r_{1}}|_{\partial\Omega_{r_{1}}}=0$, so 
$$\frac{\partial v}{\partial r}|_{r=r_{1}}\leqslant \frac{\partial u_{r_{1}}}{\partial r}|_{r=r_{1}}$$

which implies \eqref{11}.

For $n\geqslant 4$, consider
$$v=C_{1}-C_{2}\frac{1}{r^{n-3+\tau}}$$

Defined on $\Omega_{r_{1}}\setminus \Omega_{r_{0}}$. $C_{1},C_{2}$ to be chosen so that

\begin{center}
$v|_{r=r_{0}}=\Lambda_{r_{0}},v|_{r=r_{1}}=1$
\end{center}

So $(v-u_{r_{1}})|_{\partial (\Omega_{r_{1}}\setminus \Omega_{r_{0}}) }\geqslant 0$.
Note that $m=0$, so for a $h(r)$
$$\triangle h(r)=h''(r)+\frac{n-2}{r}h'(r)+h''(r)O(r^{-\mu})+h'(r)O(r^{-1-\mu})$$

Calculation show
\begin{equation}
\triangle v=-C_{2}\tau(n-3+\tau)\frac{1}{r^{n-1+\tau}}+O(\frac{1}{r^{n-2+\tau+\mu}})
\end{equation}

And
$$
C_{2}=\frac{1-\Lambda_{r_{0}}}{\frac{-1}{r_{1}^{n-3+\tau}}+\frac{1}{r_{0}^{n-3+\tau}}}$$

Fix $r_{0}$ large, and let $r_{1}$ large, then exists $\delta>0$,

\begin{center}
$C_{2} \geqslant \delta$
, and $\triangle v <0$ on $\Omega_{r_{1}}\setminus \Omega_{r_{0}}$
\end{center}

Then use maximum principle, and same argument show \eqref{22}.
\end{proof}

\begin{proposition}\label{ALFRigidity1}

Let $(M^n,g)$ be an ALF manifold with scalar curvature $R_g\geqslant 0$ and $\mu$-condition, $n\leqslant 7 $ and its $\mathbb S^1$-factor at the infinity represents a non-trivial element in $\pi_1(M)$, if $\beta =1$, in $n=3$, then $R_{g}\equiv 0$ ; if $m=0$, in $n\geqslant 4$, then $R_{g}\equiv 0$.

\end{proposition}

\begin{proof}If not, $R(p)>0$, $ u_{r}$ defined in \eqref{u_r}, consider
$$g_{r_{1}}=u_{r_{1}}^{\frac{4}{n-2}}g$$

defines a new metric on $\Omega_{r_{1}}$, and we assume $r_{1}$ large. By the definition of $u_{r_{1}}$, 
\begin{center}
 $R_{g_{r_{1}}}\geqslant 0$ on $\Omega_{r_{1}}$, and $R_{g_{r_{1}}}(p)>0$. 
\end{center}

The mean curvature of $g_{r_{1}}$ on $\partial \Omega_{r_{1}}$ is

\begin{equation}
H_{g_{r_{1}}}(r_{1})=\frac{n-2}{8(n-1)}u_{r_{1}}^{-\frac{4}{n-2}}\partial_{N}u_{r_{1}}|_{r=r_{1}}+u_{r_{1}}^{-2}H_{g}(r_{1})=\frac{n-2}{8(n-1)}\partial_{N}u_{r_{1}}|_{r=r_{1}}+H_{g}(r_{1})
\end{equation}

$N$ is the outer unit normal vector, and close to $\partial_{r}$.

\textbf{Case 1} $n=3$:

Follow from \eqref{meancurvature1} and Lemma $\ref{Ur}$, we have

\begin{equation}\label{esitimate H}
H_{\tilde{g}}^{+}(r_1)=H_{g_{r_{1}}}(r_{1})\geqslant\frac{1}{r_{1}}+\frac{\delta}{r_{1}\log r_{1}}
\end{equation}

here $\delta$ be positive constant, and may change from line to line. As before, we define a new metric on $M^{3}$,
\begin{equation}\label{newg}
\tilde{g}=\left\{
\begin{aligned}
g_{r_{1}}, &\quad \text{if $r\leqslant r_{1}$},\\
u^2dr^2+
r^2 d\phi^2 +d\theta^2+[1-\frac{1}{\lambda}(r-r_{0})]\sigma_{r}, &\quad \text{if $r_1\leqslant r\leqslant  r_1+\lambda$},\\
dx^{2}+d\theta^2,& \quad \text{if $r_1+\lambda\leqslant r $}
\end{aligned}
\right.
\end{equation}

Here choose $\lambda=1$, $u$ satisfied \eqref{PSCExtensionQSeq}. By \eqref{bargn=3} and \eqref{esitimate H}, for $r_{1}$ large, we have
$$u(r_{1},\cdot)\leqslant 1-\frac{\delta}{\log r_{1}}$$

From \eqref{Mestimate}, $M\leqslant C r^{-1}_1$ here, then
$$u(r_1+\lambda,\cdot)\leqslant 1-\frac{\delta}{\log r_{1}}$$

then
\begin{equation}
\begin{aligned}
H_{\tilde{g}}^{+}(r_1+\lambda)=\frac{H_{\bar{g}}(r_{1}+\lambda)}{u(r_{1}+\lambda,\cdot)}&\geqslant \frac{1}{r_{1}+\lambda}+\frac{\delta}{(r_{1}+\lambda)\log (r_{1}+\lambda)}\\
&>\frac{1}{(r_1+\lambda)}=H_{\tilde{g}}^{-}(r_1+\lambda)
\end{aligned}
\end{equation}

then same discussion follow the inequality part proof of Theorem \ref{ALFPMT}, contradicted to Theorem \ref{Thm: homotopy nontrivial no PSC}.

\textbf{Case 2} $n\geqslant4$:

The new metric $\tilde{g}$ on $M^{n}$ defined in the same form of \eqref{newg}, choose $\tau \ll 1$, and let  
$$\lambda=r_{1}^{-\epsilon}, n-4+\tau<\epsilon<\mu-(n-2-\tau)$$

this can be done because $\mu>2n-6$. Follow from \eqref{meancurvature1} and Lemma $\ref{Ur}$, we have
\begin{equation}\label{esitimate H2}
H_{\tilde{g}}^{+}(r_1)=H_{g_{r_{1}}}(r_{1})\geqslant\frac{n-2}{r_{1}}+\frac{\delta}{r_{1}^{n-2+\tau}}
\end{equation}

here $\delta$ be positive constant, and may change from line to line .By \eqref{bargn=4} and \eqref{esitimate H2}, for $r_{1}$ large, we have

$$u(r_{1},\cdot)=\frac{H_{\bar{g}}(r_{1})}{H_{\tilde{g}}^{+}(r_{1})}\leqslant 1-\frac{\delta}{r_{1}^{n-3+\tau}}+o(r_{1}^{-(n-3+\tau)})$$

by $M\leqslant C r_{1}^{-1}$, we have
$$u(r_{1}+\lambda,\cdot)\leqslant 1-\frac{\delta}{r_{1}^{n-3+\tau}}$$

hence

\begin{equation}
\begin{aligned}
H_{\tilde{g}}^{+}(r_1+\lambda)=\frac{H_{\bar{g}}(r_{1}+\lambda)}{u(r_{1}+\lambda,\cdot)}&\geqslant \frac{n-2}{r_{1}+\lambda}+\frac{\delta}{(r_{1}+\lambda)^{n-2+\tau}}\\
&>\frac{n-2}{r_{1}+\lambda}= H_{\tilde{g}}^{-}(r_1+\lambda)
\end{aligned}
\end{equation}

then same discussion follow the inequality part proof of Theorem \ref{ALFPMT}, contradicted to Theorem \ref{Thm: homotopy nontrivial no PSC}.

\end{proof}

\begin{lemma}\label{RicciflowkeepALF}

Assume $(M^n,g)$ be an ALF manifold with $\beta=1$ if $n=3$, or $m=0$, in $n\geqslant 4$, i.e $g=dx^{2}+d\theta^2 +\mathbb \sigma$, moreover, we assume $\sum\limits^4_{k=3}\sum\limits_{|\alpha|=k}r^k|\partial^\alpha \sigma|=O(r^{-\mu})$. Consider Ricci flow

\begin{equation}\label{ricciflow}
\left\{
\begin{aligned}
\partial_{t}g(t)&=-2Ric(t)\\
g(0)&=g,
\end{aligned}
\right.
\end{equation}

then extsts $T>0$, \eqref {ricciflow} have smooth solution in $[0,T]$, and $(M,g(t))$, $t\in(0,T]$ is still ALF  with $\beta=1$ if $n=3$ ; with $m=0$ if $n\geqslant4$.

\end{lemma}
\begin{proof}By Theorem 1 of Shi-WanXiong\cite{1989Deforming}, there exists a positive number $T$, the solution of \eqref{ricciflow} exist for $t\in [0,T]$. Let $Rm$ be the curvature tenser of $g(t)$, directly we have
$$Rm|_{t=0}=O(r^{-\mu-2}),\nabla Rm|_{t=0}=O(r^{-\mu-3}),\nabla^{2} Rm|_{t=0}=O(r^{-\mu-4})$$

Now,we are going to proof $\sum\limits^2_{k=0}\sum\limits_{|\alpha|=k}r^k|\partial^\alpha (g(t)-g(0))|=O(r^{-\mu})$, $t\in (0,T]$, then the Lemma was established .

Step 1 : $|Rm|\leqslant C$, $|\nabla Rm|\leqslant C$, $|\nabla^{2} Rm|\leqslant C$, $t\in [0,T]$, $C$ be a positive constant depends only on $g$, and may change from line to line, $|\cdot|$ is the tensor norm here.

From \cite{1989Deforming}, we have
$$|Rm|\leqslant C, C^{-1}g(0)\leqslant g(t)\leqslant Cg(0)$$

To estimate $|\nabla Rm|$, let
$$\varphi (x,t)=\xi|\nabla Rm|^{2}$$

where $\xi$ from \cite{1989Deforming} section7.(19), has compact support, satisfied (20)(56) in section7 of \cite{1989Deforming}. Direct calculation
 show
$$\partial_{t}\varphi\leqslant \Delta \varphi+(-\Delta \xi+C\frac{|\nabla \xi|^{2}}{\xi}+C\xi)|\nabla Rm|^{2}$$

by(20)(56),
$$\partial_{t}\varphi\leqslant \Delta \varphi+C_{2}|\nabla Rm|^{2}$$

meantime we have
$$\partial_{t}|Rm|^{2}\leqslant\Delta |Rm|^{2}-2|\nabla Rm|^{2}+C_{1}$$

then choose constant $a$ large enough, then
$$\partial_{t}(a|Rm|^{2}+\varphi)\leqslant\Delta(a|Rm|^{2}+\varphi)+C_{1}a$$

Use maximum principle in compact domain, and fact that $|Rm|\leqslant C$, $(a|Rm|^{2}+\varphi)|_{t=0}\leqslant C$, we have $\varphi \leqslant C,t\in [0,T]$, hence $|\nabla Rm|\leqslant C$. Similarly, consider $\partial_{t}(a|\nabla Rm|^{2}+\xi |\nabla^{2}Rm|^{2})$, we get $|\nabla^{2}Rm|\leqslant C$.

Step 2 : $\partial g(t)=O(1),t\in [0,T] $.

For any fixed $x\in M^{n}$, let 
$$\varphi_{0}(t)=|\partial (g(x,t)-g(x,0))|^{2}$$

we have
 $$\partial_{t}\varphi_{0}\leqslant C(|\nabla Rm|^{2}+\varphi_{0}|Rm|^{2}+\varphi_{0}|^{2})\leqslant C+C\varphi_{0} $$

this means  $\varphi_{0}(t)\leqslant C(\varphi_{0}(0)+1)\leqslant C$, i.e $\partial g=O(1)$.

Step 3 : $|Rm|=O(r^{-\mu-2})$, $|\nabla Rm|=O(r^{-\mu-3})$, $|\nabla^{2} Rm|=O(r^{-\mu-4})$, $ t\in [0,T]$.

Let
$$\varphi_{1}(x,t)=r^{\alpha}|Rm|^{2},\alpha=2\mu+4$$

so $\varphi_{1}(0)\leqslant C$, then we show
\begin{equation}\label{maxvarphi}
\partial_{t}\varphi_{1}(x,t)\leqslant \Delta \varphi_{1}(x,t)+C\varphi_{1}(x,t)
\end{equation}

due to step 1, use maximum principle of noncompact manifolds (see Theorem 7.42 of \cite{Chow2006Hamilton}), we get $\varphi_{1}(x,t)\leqslant C$, so $|Rm|=O(r^{-\mu-2})$, $t\in [0,T]$. To get \eqref{maxvarphi}, by a directc computation, we have
\begin{displaymath}
\begin{aligned}
\partial_{t}\varphi_{1}&=r^{\alpha}\partial_{t}|Rm|^{2}=r^{\alpha}(\Delta|Rm|^{2}-2|\nabla Rm|^{2}+Rm\ast Rm\ast Rm)\\
\Delta(r^{\alpha}|Rm|^{2})&=r^{\alpha}\Delta|Rm|^{2}+2\alpha r^{\alpha-1}\nabla r \nabla|Rm|^{2}+(\Delta r^{\alpha})|Rm|^{2}\\
\end{aligned}\end{displaymath}

hence
$$\partial_{t}\varphi_{1}=\Delta \varphi_{1}-2\alpha r^{\alpha-1}\nabla r \nabla|Rm|^{2}-(\Delta r^{\alpha})|Rm|^{2}-2r^{\alpha}|\nabla Rm|^{2}+r^{\alpha}Rm\ast Rm\ast Rm$$

Estimate the right terms, see that $|\nabla |Rm|^{2}|\leqslant 2|\nabla Rm||Rm|$, by Cauchy-Schwarz inequality,
$$
-2\alpha r^{\alpha-1}\nabla r \nabla|Rm|^{2}-2r^{\alpha}|\nabla Rm|^{2}\leqslant Cr^{\alpha-2}|\nabla r|^{2}|Rm|^{2}\leqslant C \varphi_{1}
$$

By step 2, 
$$\Delta r^{\alpha}=O(1)\partial\partial r^{\alpha}+O(1)\partial r^{\alpha}=O(1)r^{\alpha-1}$$

so 
$$(\Delta r^{\alpha})|Rm|^{2}\leqslant C \varphi_{1}$$

And by $|Rm|=O(1)$, we have
$$r^{\alpha}Rm\ast Rm\ast Rm\leqslant C\varphi_{1}$$

thus we get \eqref{maxvarphi}.
Let 
$$\varphi_{2}=r^{\beta}|\nabla Rm|^{2}, \beta=2\mu+6, \varphi_{3}=r^{\gamma}|\nabla^{2}Rm|^{2}, \gamma=2\mu+8$$

then same argument, we get the control of $|\nabla Rm| $ and $|\nabla^{2}Rm|$.

Step 4 : $\sum\limits^2_{k=0}\sum\limits_{|\alpha|=k}r^k|\partial^\alpha (g(t)-g(0))|=O(r^{-\mu})$, $t\in (0,T]$.

Integral Ricci flow at a point, 
$$g(t)-g(0)=\int_{0}^{t}-2Ric(s) ds$$
the norm of the right term was controlled by $C|Rm|$, take the derivative of both sides, follow step 3, we complete the proof.

\end{proof}

\begin{theorem}\label{ALFRigidity2}

Let $(M^n,g)$ be an ALF manifold with scalar curvature $R_g\geqslant 0$ and $\mu$-condition, and its $\mathbb S^1$-factor at the infinity represents a non-trivial element in $\pi_1(M)$, moreover, we assume $\sum\limits^4_{k=3}\sum\limits_{|\alpha|=k}r^k|\partial^\alpha \sigma|=O(r^{-\mu})$. If $\beta =1$ in $n=3$, or $m=0$ in $n\geqslant 4$, then $g$ is Ricci flat.

\end{theorem}

\begin{proof}By Proposition\ref{ALFRigidity1}, $R_{g} \equiv 0$, if $Ric_{g}$ not flat, by Lemma\ref{RicciflowkeepALF}, through Ricci flow, we get $(M^{n},g(t)),t\in(0,T]$, and is still ALF, with $\beta =1$ in $n=3$, or $m=0$ in $n\geqslant 4$. The scalar curvature satisfied
$$\partial_{t}R=\Delta R+2|Ric|^{2}$$

by the maximum principle of noncompact manifolds (see Theorem 7.42 of \cite{Chow2006Hamilton}), $R\geqslant 0$. If $Ric_{g}$ not flat, we have $R>0$, for $t>0$, which contradicted to Theorem\ref{ALFRigidity1}.

\end{proof}
\begin{proof}[Proof of the rigidity part of Theorem \ref{ALFPMT}]:
When $\beta =1$ in $n=3$, or $ m=0$ in $n\geqslant 4$, by Theorem \ref{ALFRigidity2}, we know $g$ is Ricci flat. When $n\geqslant 4$, follow the work of   Minerbe \cite{minerbe2008a}, our situation is  the example 1 in Page 929 of \cite{minerbe2008a}, as our order of $\sigma$ in hypothesis is large enough (in fact we only need $\mu>n-3$ here), we get $\mu_{g}^{GB}=0$ in his definition, by Theorem 1 of \cite{minerbe2008a}, $(M^{n},g)$ is flat and isometric to $\mathbb{R}^{n-1}\times \mathbb{S}^{1}$.
When $n=3$, see the appendix of \cite{MR2855540}, follow the proof of Theorem 1 in \cite{minerbe2008a}, as we assume $\mu>1$, there are $2$ linear independent $g$-parallel one-forms that are asymptotic to $dx^{1}$, $dx^{2}$ at the infinity, which implies $(M^{3},g)$ is flat and isometric to $\mathbb{R}^{2}\times \mathbb{S}^{1}$.

\end{proof}
\begin{remark}
To proof the rigidity part for ALG situation, same argument, we get that $g$ is Ricci flat, then by the appendix in \cite{MR2855540} and follow the proof of Theorem 1 in \cite{minerbe2008a}, for the order hypothesis of $\sigma$ here, there are $n-2$ linear independent $g$-parallel one-forms that are asymptotic to $\{dx^{i}\}_{i=1}^{n-2}$ at the infinity , which implies $ M^{n}$ isometric to $\mathbb{R}^{n-2}\times E$. $E$ is a compact $2$-manifold with flat metric, hence must be flat $\mathbb{T}^{2}$ with area $4\pi^{2}$.
\end{remark}

\section{Application to fill-in problems}\label{Sec: fill-in}
First we recall some notions for convenience. Given an orientable closed Riemannian manifold $(\Sigma,\gamma)$, an it admissible fill-in of $(\Sigma,\gamma)$ is a compact orientable Riemannian manifold $(\Omega,g)$ with non-negative scalar curvature and mean convex boundary such that the boundary is isometric to $(\Sigma,\gamma)$. The supremum total mean curvature of $(\Sigma,\gamma)$ with respect to admissible fill-ins is defined to be
$$
\Lambda_+(\Sigma,\gamma)=\sup\left\{\int_{\partial\Omega}H_{\partial\Omega}\,\mathrm d\sigma_g:\text{ $(\Omega,g)$ is an admissible fill-in of $(\Sigma,\gamma)$}\right\},
$$
where $H_{\partial\Omega}$ is the mean curvature of boundary $\partial\Omega$ with respect to the outer unit normal and $\mathrm d\sigma_g$ is the area element of $\partial\Omega$.

We are going to prove
\begin{theorem}\label{Thm: re fill-in 1}
Let $\Sigma_0$ be a convex hypersurface or curve in the Euclidean space $\mathbb E^n$ with total mean curvature $T_0$. If $(\Omega,g)$ is an admissible fill-in of the product manifold $\Sigma_0\times \mathbb S^1(l)$ such that the circle component is homotopically non-trivial in $\Omega$, then for $n\leq 6$ it holds
\begin{equation*}
\int_{\partial\Omega}H_{\partial\Omega}\,\mathrm d\sigma_g\leq 2\pi l T_0,
\end{equation*}
where $H_{\partial\Omega}$ is the mean curvature of $\partial\Omega$ with respect to the unit outer normal and $\mathrm d\sigma_g$ is the area element of $\partial\Omega$ with the induced metric. If $(\Omega,g)$ is an admissible fill-in of $\Sigma_0\times \mathbb S^1(l)$ with the equality above, then $(\Omega,g)$ is static with vanishing scalar curvature.
\end{theorem}

From Theorem \ref{Thm: re fill-in 1} we have the following corollary:
\begin{corollary}
Let $(\Sigma,\gamma)$ be the flat $2$-torus $\mathbb S^1(l_1)\times \mathbb S^1(l_2)$ with $l_1\leq l_2$. Then it holds
$$
\Lambda_+(\Sigma,\gamma)= 4\pi^2 l_2.
$$
Moreover, if $(\Omega,g)$ is an admissible fill-in of $(\Sigma,\gamma)$ whose boundary has total mean curvature $4\pi^2l_2$, then it is static with vanishing scalar curvature.
\end{corollary}

\begin{proof}
This is a direct consequence from the $3$-dimensional topology. Let $(\Omega,g)$ be an admissible fill-in of $(\Sigma,\gamma)$. If $\mathbb S^1(l_1)$ is homotopically non-trivial in $\Omega$, then it follows from Theorem \ref{Thm: re fill-in 1} that
$$
\int_{\partial\Omega}H_{\partial\Omega}\,\mathrm d\sigma_g\leq 4\pi^2l_2.
$$
Otherwise $\mathbb S^1(l_1)$ is homotopic to a point in $\Omega$. From Dehn's lemma we can deduce that $\Omega$ is in the form of $\mathbb D^2\times \mathbb S^1\sharp \Omega'$ for an orientable closed $3$-manifold $\Omega'$, where $\mathbb S^1(l_1)$ is the boundary of $\mathbb D^2$ and $\mathbb S^1(l_2)$corresponds to the $\mathbb S^1$-component. It turns out that $\mathbb S^1(l_2)$ is homotopically non-trivial in $\Omega$. As a result of Theorem \ref{Thm: re fill-in 1} we have
$$
\int_{\partial\Omega}H_{\partial\Omega}\,\mathrm d\sigma_g\leq 4\pi^2l_1\leq 4\pi^2l_2.
$$
On the other hand, the product manifold $\mathbb D^2(l_1)\times \mathbb S^1(l_2)$ provides an admissible fill-in of $(\Sigma,\gamma)$ whose boundary has total mean curvature $4\pi^2 l_2$ and so it follows $\Lambda_+(\Sigma,\gamma)=4\pi^2l_2$. The last statement comes from that of Theorem \ref{Thm: re fill-in 1}.
\end{proof}
In the following, we divide the proof of Theorem \ref{Thm: re fill-in 1} into two cases depending on dimensions.
\begin{proof}[Proof of Theorem \ref{Thm: re fill-in 1} when $n=2$]
In this case, $\Sigma$ is isometric to the flat $2$-torus $\mathbb S^1(l_1)\times \mathbb S^1(l)$ and the total mean curvature $T_0$ of $\Sigma_0$ is always $2\pi$. From the polar coordinate of $\mathbb E^2$ we have the diffeomorphism
$$
\Phi:\mathbb S^1\times \mathbb S^1\times[0,+\infty)\to (\mathbb E^2-\mathbb D^2(l_1))\times \mathbb S^1(l)
$$
such that the pullback metric can be written as
$
\bar g=(l_1+t)^2\mathrm d\theta_1^2+l^2\mathrm d\theta_2^2+\mathrm dt^2
$, where $\mathbb D^2(l_1)$ is the disk with radius $l_1$ in $\mathbb E^2$ and $\mathrm d\theta_i$ is the length element of unit circle $\mathbb S^1$. With a smooth positive function $u_0$ on $\mathbb S^1\times \mathbb S^1$ to be determined later, we consider the quasi-spherical metric equation
\begin{equation}\label{Eq: quasispherical 1}
\frac{1}{t}\frac{\partial u}{\partial t}=u^2\Delta_{\gamma_t}u,\quad u(\cdot,0)=u_0,
\end{equation}
where $\gamma_t$ is the induced metric of the hypersurface $\mathbb S^1\times \mathbb S^1\times\{t\}$. The geometric meaning of above equation is that the associated metric $\tilde g=(l_1+t)^2\mathrm d\theta_1^2+l^2\mathrm d\theta_2^2+u^2\mathrm dt^2$ on $\mathbb S^1\times \mathbb S^1\times[0,+\infty)$ has vanishing scalar curvature.

Now we analyze equation \eqref{Eq: quasispherical 1}. From parabolic maximum principle, it is easy to see the a priori estimate
\begin{equation}\label{Eq: u bounded 1}
\inf_{\mathbb S^1\times \mathbb S^1} u_0\leq u\leq \sup_{\mathbb S^1\times\mathbb S^1}u_0
\end{equation}
and so the solution $u$ exists all the time. We are going to show that $u(\cdot,t)$ converges to a constant as $t$ tends to infinity. For our purpose, given any positive integer $k$ we define the following map
$$
\Psi_k: \mathbb S^1\times \mathbb S^1\times [0,4]\to \mathbb S^1\times \mathbb S^1\times [k,5k],\quad (\theta_1,\theta_2,\tau)\mapsto (\theta_1,k\theta_2,(\tau+1)k)
$$
and $v_k=u\circ\Psi_k$. Then we can compute
\begin{equation}\label{Eq: solution v_k 1}
\frac{\partial v_k}{\partial\tau}=v_k^2\left(\frac{(\tau+1)k^2}{\left(l_1+(\tau+1)k\right)^2}\frac{\partial^2 v_k}{\partial\theta_1^2}+\frac{\tau+1}{l^2}\frac{\partial^2 v_k}{\partial\theta_2^2}\right)\quad \text{on}\quad \mathbb S^1\times \mathbb S^1\times [0,4].
\end{equation}

Denote
$$
M(t)=\sup_{\mathbb S^1\times \mathbb S^1\times\{t\}} u\quad \text{and}\quad
m(t)=\inf_{\mathbb S^1\times \mathbb S^1\times\{t\}} u.
$$
The parabolic maximum principle implies that $M(t)$ is monotone decreasing and $m(t)$ is monotone increasing as $t$ increases. As a consequence, both the limits
$$M=\lim_{t\to+\infty}M(t)\quad \text{and}\quad m=\lim_{t\to+\infty}m(t)$$
exist. Notice that $u(\cdot,t)$ converges to a constant as $t$ tends to infinity if and only if it holds $M=m$. We just need to rule out the possibility that there is an increasing sequence $t_j\to +\infty$ such that the oscillations $O_j=M(t_j)-m(t_j)$ are greater than a fixed positive constant $\epsilon_0$. For each $t_j$ we can find a positive integer $k_j$ such that $t_j$ belongs to $[2k_j,4k_j]$. Then $v_{k_j}$ is a sequence of solution to \eqref{Eq: solution v_k 1} with
\begin{equation}\label{Eq: oscillation 1}
\mathop{\rm osc}\limits_{\mathbb S^1\times \mathbb S^1\times[1,3]} v_{k_j}\geq \epsilon_0>0.
\end{equation}
From the definition of $v_k$ and estimate \eqref{Eq: u bounded 1} we see that $v_{k_j}$ are uniformly bounded from below and above. Standard theory for parabolic PDE equations (see H\"older estimate \cite[P.419, Theorem 1.1]{LSU1968} and interior Schauder estimate \cite[P. 92, Theorem 1]{Fri1964}) implies that $v_{k_j}$ converges smoothly to a limit function $v$ up to a subsequence satisfying
$$
\frac{\partial v}{\partial\tau}=v^2\left(\frac{1}{\tau+1}\frac{\partial^2 v}{\partial\theta_1^2}+\frac{\tau+1}{l^2}\frac{\partial^2 v}{\partial\theta_2^2}\right)\quad \text{on}\quad \mathbb S^1\times \mathbb S^1\times [1,3].
$$
Clearly \eqref{Eq: oscillation 1} holds for $v$ as well. On the other hand, $v$ must attain its minimum value $m$ in some interior point since
$$
\min_{\mathbb S^1\times \mathbb S^1\times \{2\}} v=\lim_{j\to+\infty} \min_{\mathbb S^1\times \mathbb S^1\times \{2\}} v_{k_j}=\lim_{j\to+\infty} m(3k_j)=m.
$$
The strong parabolic maximum principle yields that $v$ is a constant function, which contradicts to \eqref{Eq: oscillation 1}.
From equation \eqref{Eq: quasispherical 1}, we have
$$
\frac{\mathrm d}{\mathrm dt}\int_{\mathbb S^1(l_1+t)\times \mathbb S^1(l)}u^{-1}(\cdot,t)\,\mathrm d\sigma=\frac{1}{l_1+t}\int_{\mathbb S^1(l_1+t)\times \mathbb S^1(l)}u^{-1}(\cdot,t)\,\mathrm d\sigma
$$
and so
$$
\int_{\mathbb S^1(l_1+t)\times \mathbb S^1(l)}u^{-1}(\cdot,t)\,\mathrm d\sigma=\frac{l_1+t}{l_1}\int_{\mathbb S^1(l_1)\times \mathbb S^1(l)}u_0^{-1}\,\mathrm d\sigma.
$$
As a result, we obtain
$$
\lim_{t\to+\infty}u(\cdot,t)=\frac{4\pi^2 l}{l_1}\left(\int_{\mathbb S^1(l_1)\times \mathbb S^1(l)}u_0^{-1}\,\mathrm d\sigma\right)^{-1}.
$$

Suppose that there is an admissible fill-in $(\Omega,g)$ of $\Sigma$ with
\begin{equation}\label{Eq: large total mean curvature 1}
\int_{\partial\Omega} H_{\partial\Omega}\,\mathrm d\sigma > 4\pi^2 l .
\end{equation}
By identifying $\partial\Omega$ and $\Sigma$, we can view $H_{\partial\Omega}$ as a smooth positive function on $\Sigma$. In this case, we set $u_0=l_1^{-1}H_{\partial\Omega}^{-1}$. From quasi-spherical metric equation \eqref{Eq: quasispherical 1} we can construct a metric $\tilde g$ on $\mathbb S^1\times \mathbb S^1\times[0,+\infty)$ with vanishing scalar curvature, where the boundary is isometric to $\Sigma$ and has mean curvature $-H_{\partial\Omega}$ with respect to the unit outer normal. The condition \eqref{Eq: large total mean curvature 1} ensures that the limit of $u$ is less than one and so the hypersurface $\mathbb S^1\times \mathbb S^1\times\{t\}$ isometric to $\mathbb S^1(l_1+t)\times \mathbb S^1(l)$ has mean curvature greater than $(l_1+t)^{-1}$ for sufficiently large $t$. Fix one such $t$ and choose $r>l_1+t$. Now we have three Riemannian manifolds: $$(\Omega,g),\quad (\mathbb S^1\times \mathbb S^1\times[0,t],\tilde g),\quad\text{and}\quad\left(\mathbb E^2/(r\mathbf Z)^2-\mathbb D^2(l_1+t)\right)\times \mathbb S^1(l).$$
We glue them together along two pairs of boundaries. It is easy to see that the underlying manifold after gluing satisfies the hypothesis of Theorem \ref{Thm: homotopy nontrivial no PSC}. On the other hand, since each of these manifolds has non-negative scalar curvature and the sum of mean curvatures on two sides of the glued boundaries is non-negative everywhere and positive somewhere, we can construct a PSC metric through a mollification procedure from \cite{Miao2002}, which leads to a contradiction.

The argument for the last statement is similar to that of \cite[Theorem 4.2]{SWY2019}. Once the scalar curvature is positive somewhere in $(\Omega,g)$, then we can decrease the scalar curvature a little bit and increase the mean curvature of the boundary without affecting the induced metric after an appropriate conformal deformation. Then we obtain a new admissible fill-in such that \eqref{Eq: large total mean curvature 1} holds but this is impossible. As a result, $(\Omega,g)$ must have vanishing scalar curvature. This also yields the static property of $(\Omega,g)$ due to \cite[Theorem 1]{Corvino2000}.
\end{proof}

\begin{proof}[Proof for Theorem \ref{Thm: re fill-in 1} when $n\geq 3$]
The proof in this case is similar to the previous one but the analysis turns out to be a little bit complicated since we cannot always reduce the problem to the standard sphere case and also we need a more careful analysis on $u$.

Let us work with an arbitrary convex hypersurface $\Sigma_0$ in $\mathbb E^n$ and denote
$$
\Phi:\Sigma_0\times[0,+\infty)\to \mathbb E^n-\Omega_0
$$
to be the outward geodesic flow with initial data $\Sigma_0$, where $\Omega_0$ is the region enclosed by $\Sigma_0$. It is clear that the Euclidean metric can be written as $g_{euc}=\mathrm dt^2+\gamma_t$.
Correspondingly, the map
$$\bar\Phi :\Sigma_0\times \mathbb S^1\times[0,+\infty)\to (\mathbb E^n-\Omega_0)\times \mathbb S^1(l)$$
is a diffeomorphism and the pullback metric can be written as $\bar g=\mathrm dt^2+\bar\gamma_t$, where $\bar\gamma_t=\gamma_t+l^2\mathrm d\theta^2$.
As before, we consider the quasi-spherical equation
\begin{equation}\label{quasispherical 2}
\bar H_t\frac{\partial u}{\partial t}=u^2\Delta_{\bar\gamma_t}u+\frac{1}{2}R_{\bar\gamma_t}(u-u^3),\quad u(\cdot,0)=u_0,
\end{equation}
where $\bar H_t$ and $R_{\bar\gamma_t}$ are mean curvature and scalar curvature of $\bar\Sigma_t$, and $u_0$ is a positive function on $\bar\Sigma_0$ to be determined later. The associated metric $\tilde g=u^2\mathrm dt^2+\bar\gamma_t$ has vanishing scalar curvature. Notice that $\Sigma_t$ is convex for all $t$, so we see $R_{\bar\gamma_t}=R_{\gamma_t}\geq 0$. The parabolic maximum principle then yields the a priori estimate
$$
\min\left\{1,\inf_{\bar\Sigma_0}u_0\right\}\leq u \leq \max\left\{1,\max_{\bar\Sigma_0}u_0\right\}
$$
and the solution to \eqref{quasispherical 2} exists all the time. Since the mean curvature and scalar curvature of $\bar\Sigma_t$ are completely determined by those of $\Sigma_t$, the same argument as in \cite[Lemma 2.2]{ST2002} implies
\begin{equation}\label{Eq: u-1 bound}
|u-1|\leq Ct^{2-n},\quad \forall\,t\geq 1.
\end{equation}
Here and in the sequel, $C$ is denoted to be some universal constant $C$ independent of $t$. Let
$$
v=t^{n-2}(u-1).
$$
Then $v$ is a uniformly bounded function satisfying the following equation
\begin{equation}\label{Eq: equation v}
\bar H_t\frac{\partial v}{\partial t}=u^2\left(\Delta_{\gamma_t}v+\frac{1}{l^2}\frac{\partial^2 v}{\partial\theta^2}\right)+v\left((n-2)\frac{\bar H_t}{t}-\frac{1}{2}R_{\bar\gamma_t}u(u+1)\right),\quad u=1+\frac{v}{t^{n-2}}.
\end{equation}
Define the map
$$
\bar\Psi_k: \Sigma_0\times \mathbb S^1\times [0,4]\to \Sigma_0\times \mathbb S^1\times [k,5k],\quad (x,\theta,\tau)\mapsto (x,\theta,(\tau+1)k)
$$
and denote $v_k=v\circ \bar\Psi_k$. Then $v_k$ satisfies
$$
\frac{\partial v_k}{\partial\tau}=f_1(k,\tau,v_k)\Delta_{\hat\gamma_{k,\tau}}v_k+f_2(k,\tau,v_k)\frac{\partial ^2v_k}{\partial\theta^2}+g(k,\tau,v_k),
$$
where
$$\hat\gamma_{k,\tau}=\frac{1}{[(\tau+1)k]^2}\gamma_{(\tau+1)k},$$
$$
f_1(k,\tau,v_k)=\frac{1}{(\tau+1)^2k\bar H_{(\tau+1)k}}\left(1+\frac{v_k}{[(\tau+1)k]^{n-2}}\right)^2,
$$
$$
f_2(k,\tau,v_k)=\frac{1}{l^2k\bar H_{(\tau+1)k}}\left(1+\frac{v_k}{[(\tau+1)k]^{n-2}}\right)^2,
$$
and
$$
g(k,\tau,v_k)=v_k\left(\frac{n-2}{\tau+1}-\frac{R_{\hat\gamma_{k,\tau}}}{2(\tau+1)^2k\bar H_{(\tau+1)k}}\left(1+\frac{v_k}{[(\tau+1)k]^{n-2}}\right)\left(2+\frac{v_k}{[(\tau+1)k]^{n-2}}\right)\right).
$$
From \cite[Lemma 2.1 and Lemma 2.4]{ST2002} we have
$$
k\bar H_{(\tau+1)k}\to \frac{n-1}{\tau+1},\quad R_{\hat\gamma_{k,\tau}}\to (n-1)(n-2),\quad \hat\gamma_{k,\tau}\to \gamma_{std},\quad\text{as}\quad k\to\infty,
$$
where $\gamma_{std}$ is the pullback of the standard spherical metric on $\Sigma_0$ by the Gauss map and all convergences are smooth with respect to $\tau$ and any coordinate system on $\Sigma_0$.

Denote
$$
M(t)=\sup_{\Sigma_t\times \mathbb S^1} v\quad \text{and}\quad
m(t)=\inf_{\Sigma_t\times \mathbb S^1} v.
$$
Combining \cite[Lemma 2.1]{ST2002} and \eqref{Eq: u-1 bound} we also see $|\beta(t)|\leq Ct^{-2}$ for all $t\geq 1$, where
$$
\beta(t)=\frac{n-2}{t}-\frac{R_{\bar\gamma_t}}{2\bar H_t}u(u+1).
$$
After applying parabolic maximum principle to \eqref{Eq: equation v}, we see
$$
M(t')\leq M(t)+Ct^{-1},\quad m(t')\geq m(t)-Ct^{-1},\quad \forall\,t'\geq t\geq 1.
$$
Therefore $v$ converges to a constant if and only if the oscillation of $v$ on $\bar\Sigma_t$ tends to $0$ as $t\to+\infty$. Also, $m(t)$ must have a limit as $t\to +\infty$.
Then we can repeating the argument in previous proof to conclude that $v(\cdot,t)$ converges to a constant as $t\to +\infty$.

Now we are ready to prove the theorem. Suppose that there is an admissible fill-in $(\Omega,g)$ of $\Sigma$ with
\begin{equation}\label{Eq: large total mean curvature 2}
\int_{\partial\Omega} H_{\partial\Omega}\,\mathrm d\sigma > 2\pi l T_0 .
\end{equation}
As before, we identify $\partial\Omega$ with $\Sigma$ and view $H_{\partial\Omega}$ as a smooth function on $\Sigma$. Take $$u_0=\frac{\bar H_0}{H_{\partial\Omega}}.$$
From a similar calculation as in \cite[Lemma 4.2]{ST2002}, it follows that the integral
$$
\int_{\bar\Sigma_t}\bar H_t\left( u^{-1}(\cdot,t)-1\right)\mathrm d\sigma_t
$$
is monotone increasing as $t$ increases. Combined with \eqref{Eq: large total mean curvature 2}, we see that the limit of $v$ is negative and so $u$ is less than one pointwise for $t$ large enough. A gluing argument then leads to desired contradiction. For the last statement, the proof is the same as before and we don't repeat it here.
\end{proof}

With the exactly same argument, we can also show the following
\begin{theorem}\label{Thm: re fill-in 2}
Let $\Sigma_0$ be a convex hypersurface or curve in the Euclidean space $\mathbb E^n$ with total mean curvature $T_0$. If $(\Omega,g)$ is an admissible fill-in of the product manifold $\Sigma_0\times \mathbb T^2$ with a flat $\mathbb T^2$ such that the torus component is incompressible in $\Omega$, then for $n\leq 5$ it holds
\begin{equation*}
\int_{\partial\Omega}H_{\partial\Omega}\,\mathrm d\sigma_g\leq T_0 \area(\mathbb T^2),
\end{equation*}
where $H_{\partial\Omega}$ is the mean curvature of $\partial\Omega$ with respect to the unit outer normal and $\mathrm d\sigma_g$ is the area element of $\partial\Omega$ with the induced metric. If $(\Omega,g)$ is an admissible fill-in of $\Sigma_0\times \mathbb T^2$ with the equality above, then $(\Omega,g)$ is static with vanishing scalar curvature.
\end{theorem}
\begin{proof}
The proof is almost the same except that we use the lifting property of $\mathbb T^2$ this time. We omit further details.
\end{proof}

{
At the end of this section, we present a proof for the following corollary
\begin{corollary}
Let $(\Sigma,\gamma)$ be a flat $2$-torus. Then $\Lambda_+(\Sigma,\gamma)<+\infty$.
\end{corollary}
\begin{proof}
First observe that flat metrics on $\mathbb T^2$ form a connected space. To see this, we start with an arbitrary flat metric $\gamma$ on $\mathbb T^2$. It is well-known that $(\mathbb T^2,\gamma)$ can be viewed as a quotient space $\mathbb R^2/\Gamma$, where $\Gamma$ is a lattice on $\mathbb R^2$. We can find an orientation-preserving affine transformation $\Phi:\mathbb R^2\to \mathbb R^2$ such that $\Phi(\mathbf Z^2)=\Gamma$. From \cite[Corollary 3.6]{matrixgroup2018} we can pick up a smooth family of orientation-preserving affine transformations
$\{\Phi_t\}_{0\leq t\leq 1}$ with $\Phi_0=\Phi$ and $\Phi_1=\id$. Then the pullback metric $\Phi_t^*(g_{euc})$ induces a smooth family of flat metrics $\{\gamma_t\}_{0\leq t\leq 1}$ on $\mathbb T^2$ with $\gamma_0=\gamma$ and $\gamma_1=\mathrm d\theta_1^2+\mathrm d\theta_2^2$. Notice that there is a positive constant $k$ such that the metric $\bar g=\mathrm ds^2+s^2\gamma_{\frac{s-1}{k}}$ on $[1,k+1]\times \mathbb T^2$ satisfying
\begin{equation}\label{Eq: second fundamental form}
\left\|\bar A_s-\frac{1}{s}\bar\gamma_{\frac{s-1}{k}}\right\|_{\bar\gamma_{s/k}}\leq \frac{1}{2s},\quad \bar\gamma_{\frac{s-1}{k}}:=s^2\gamma_{\frac{s-1}{k}},
\end{equation}
where $\bar A_s$ is the second fundamental form of $\{s\}\times\mathbb T^2$ with respect to $\partial_s$. The quasi-spherical equation with respect to $\bar g$ now reads
$$
\bar H_s\frac{\partial u}{\partial s}=u^2\Delta_{\bar\gamma_{\frac{s-1}{k}}}u-\frac{1}{2}R_{\bar g}u.
$$
A direct computation combined with \eqref{Eq: second fundamental form} yields
\begin{equation*}
\begin{split}
\frac{\mathrm d}{\mathrm ds}\int_{\{s\}\times \mathbb T^2}\bar H_s u^{-1}\,\mathrm d\sigma_s &=\frac{1}{2}\int_{\{s\}\times \mathbb T^2}\left(\bar H_s^2-\|\bar A_s\|^2\right)u^{-1}\,\mathrm d\sigma_s\geq \frac{1}{4s}\int_{\{s\}\times \mathbb T^2}\bar H_s u^{-1}\,\mathrm d\sigma_s.
\end{split}
\end{equation*}
Using the quasi-spherical metric $\tilde g=u^2\mathrm ds^2+s^2\gamma_{\frac{s-1}{k}}$ and gluing method (with a handle of corner as in \cite{Miao2002}), we can always construct an admissible fill-in of $\mathbb S^1(k+1)\times \mathbb S^1(k+1)$ with total mean curvature no less than $T_0/C$ once $(\mathbb T^2,\gamma)$ admits an admissible fill-in with total mean curvature $T_0$, where $C$ is a universal constant depending only on $\gamma_t$ and $k$. This gives $\Lambda_+(\mathbb T^2,\gamma)\leq 4\pi^2(k+1)C<+\infty$.
\end{proof}
}

\newpage

\appendix
\section{Some topological results}
Let $\mathbb T^n$ be a $n$-torus and $\Omega=B\times \mathbb S^1_n$ be a subregion of $\mathbb T^n$ associated with the embedding $i:\Omega\to \mathbb T^n$, where $B$ denotes a ball in $\mathbb T^{n-1}$. Take a closed curve $\gamma$ in $\partial\Omega$ such that $\gamma$ is homotopic to $\mathbb S^1_n$ in $\mathbb T^n-\Omega$.

In the following, we are going to prove
\begin{proposition}\label{Prop: same homotopy local and global}
For any $(M_2,\phi_2)$ in $\mathcal C_\Omega$, $\gamma$ is homotopic to a point in $(\mathbb T^n,i)\sharp(M_2,\phi_2)$ if and only if it is homotopic to a point in $M_2-\phi_2(\Omega)$.
\end{proposition}

Clearly, the difficulty lies in the ``only if'' part and the basic idea to prove this part is to find a covering space of $(\mathbb T^n,i)\sharp(M_2,\phi_2)$ such that the lift of curve fails to be a closed curve under the assumption that $\gamma$ is homotopiclly non-trivial in $M_2-\phi_2(\Omega)$. Since the universal covering of $M_2-\phi_2(\Omega)$ has above property in this case, we just need to extend this to a covering of $(\mathbb T^n,i)\sharp(M_2,\phi_2)$.

First let us establish several preliminary results.
\begin{lemma}
Let $M=\mathbb T^2-D$ with $D$ a disk in $\mathbb T^2$ and $x_0$ be a point in $\partial M$. Given any subgroup $G$ of $\pi_1(\partial M,x_0)$, we can find a covering $p:\hat M\to M$ such that $\partial \hat M$ consists of two components $\hat\partial_1$ and $\hat\partial_2$, where it holds $p_*\pi_{1}(\hat\partial_i,\hat x_i)=G$ for some point $\hat x_i$ in $\hat\partial_i$ for $i=1,2$.
\end{lemma}
\begin{proof}
For convenience, we consider $M$ as the quotient space of a `+' shaped stripe by identifying opposite edges $\alpha_1$ and $\beta_1$ as shown in the following figure. In this setting, the boundary is a curve $\gamma$ that starts from point $x_0$, then passes each arrows in the marked order and finally returns to point $x_0$.
\begin{figure}[htbp]
\centering
\includegraphics[width=5cm]{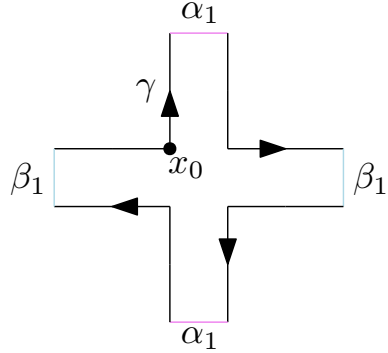}
\caption{The space $\mathbb T^2-D$}
\end{figure}

With the generator $[\gamma]$ of $\pi_1(\partial M,x_0)$ above, any non-trivial subgroup $G$ is generated by $k[\gamma]$ for some positive integer $k$. In this case, the desired cover is given by
\newpage
\begin{figure}[htbp]
\centering
\includegraphics[width=8cm]{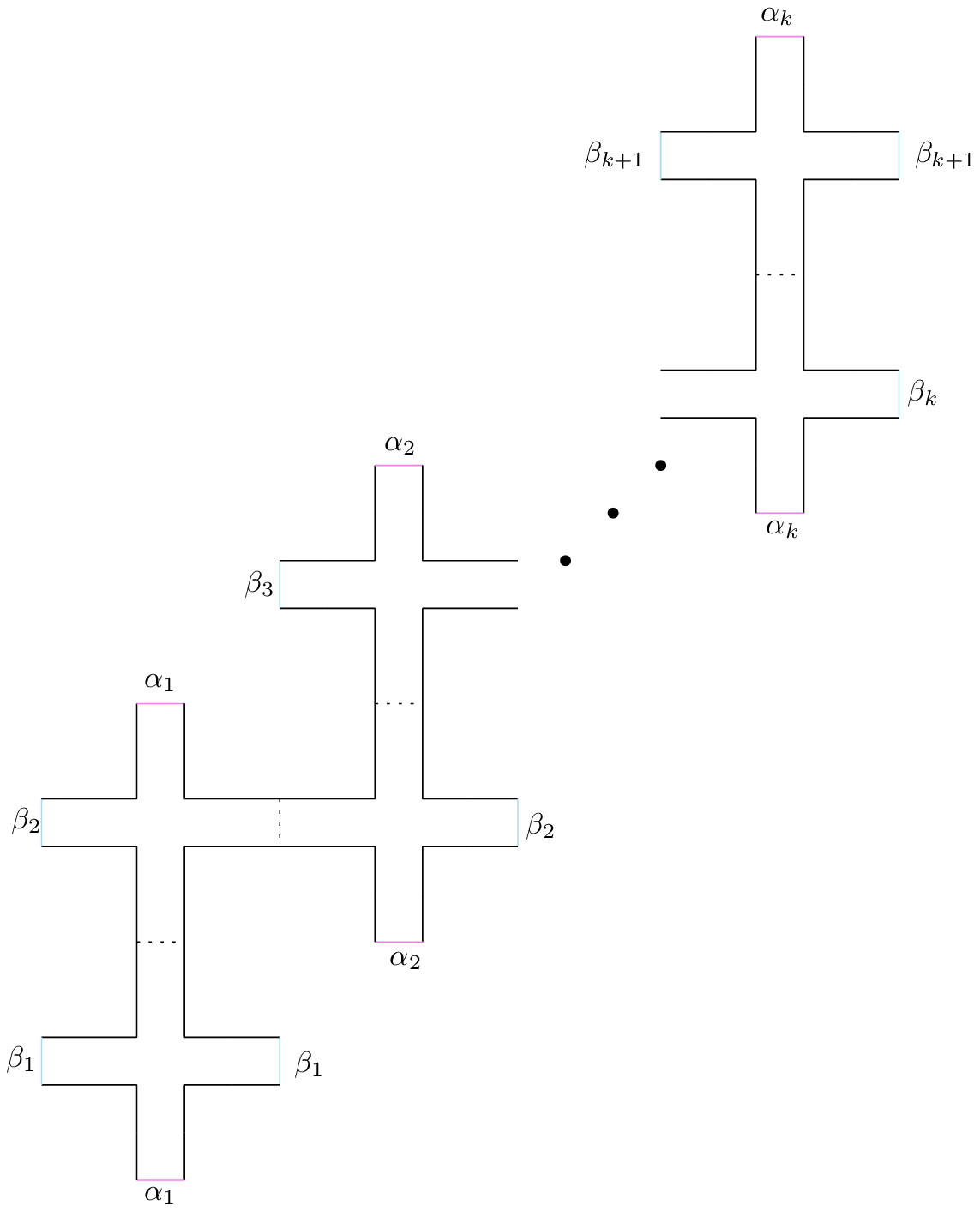}

\end{figure}

If $G$ is the trivial group, we take $k=\infty$ in above construction and the corresponding cover looks like
\begin{figure}[htbp]
\centering
\includegraphics[width=8cm]{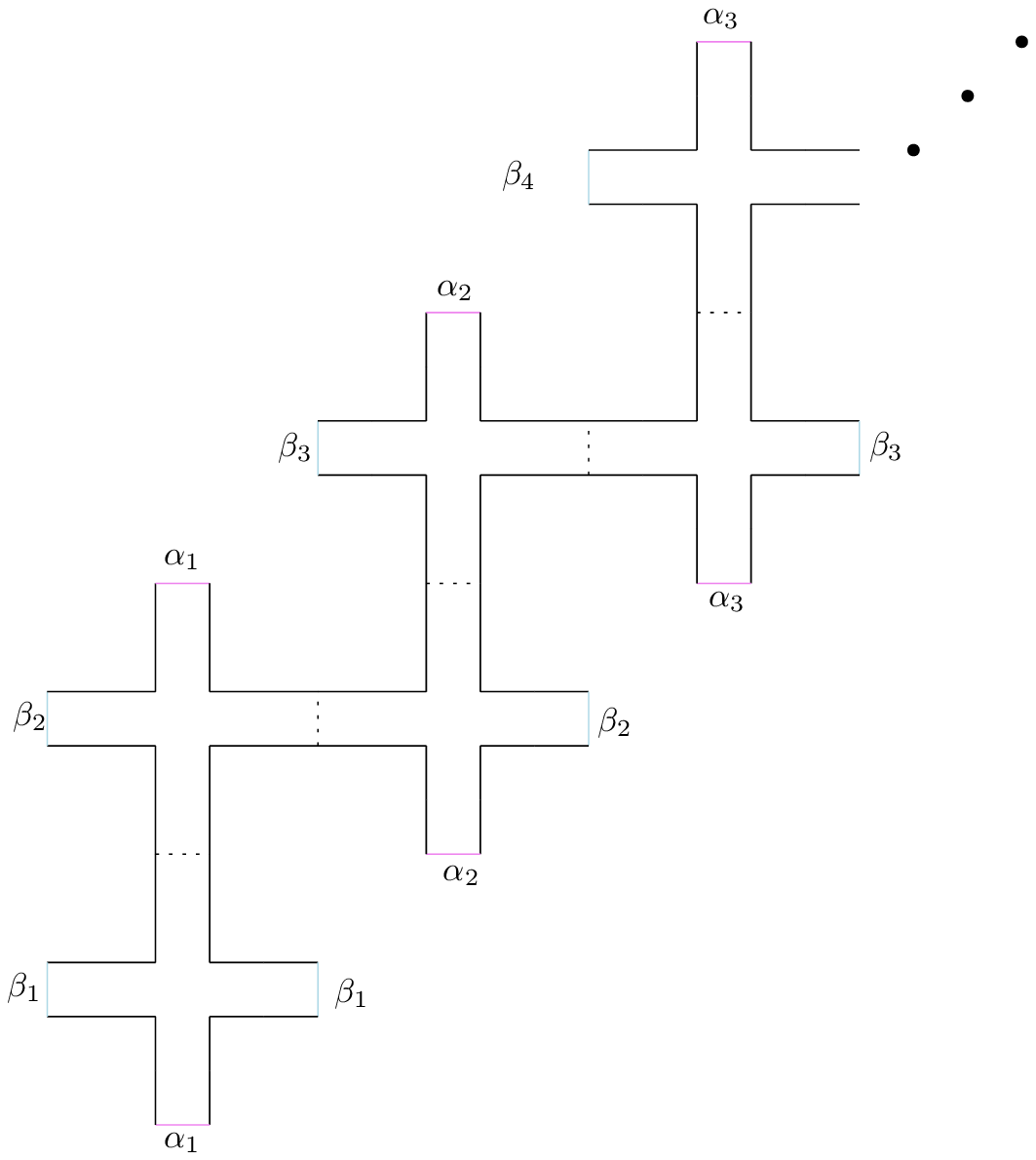}
\end{figure}

\end{proof}
\begin{corollary}\label{Cor: existence for covering}
Let $M=(\mathbb T^2-D)\times \mathbb S_3^1$ and $x_0$ be a point in $\partial M$. Given any subgroup $G$ of $\pi_1(\partial M,x_0)$, we can find a covering $p:\hat M\to M$ such that $\partial \hat M$ consists of two components $\hat\partial_1$ and $\hat\partial_2$, where it holds $p_*\pi_1(\hat\partial_i,\hat x_i)=G$ for some point $\hat x_i$ in $\hat\partial_i$ for $i=1,2$.
\end{corollary}
\begin{proof}
We write $x_0=(y_0,z_0)$ for some $y_0\in\partial D$ and $z_0\in\mathbb S^1$. The desired consequence follows from the fact $\pi_1(\partial M,x_0)\cong \pi_1(\partial D,y_0)\times \pi_1(\mathbb S^1,z_0)$.
\end{proof}

When the dimension $n$ is greater than $3$, we state the following simple fact.

\begin{lemma}\label{Lem: existence for covering D4+}
Let $M=(\mathbb T^{n-1}-B)\times \mathbb S_n^1$ and $x_0$ be a point in $\partial M$. Given any subgroup $G$ of $\pi_1(\partial M,x_0)$, we can find a covering $p:\hat M\to M$ such that $\partial \hat M$ is connected and it holds $p_*\pi_1(\partial\hat M,\hat x)=G$ for some point $\hat x$ in $\partial\hat M$.
\end{lemma}
\begin{proof}
It follows directly from the product structure of $(\mathbb T^{n-1}-B)\times \mathbb S^1_n$ and the fact that $\pi_1(\partial M,x_0)$ is the free group generated by $\mathbb S^1_n$.
\end{proof}

We are going to give a proof for Proposition \ref{Prop: same homotopy local and global}.

\begin{proof}[Proof for Proposition \ref{Prop: same homotopy local and global}]
Recall
$$
(\mathbb T^n,i)\sharp(M_2,\phi_2)=(\mathbb T^n-\Omega)\sqcup_{\Phi}(M_2-\phi_2(\Omega)),\quad \text{where}\quad \Phi=\phi_2:\partial\Omega\to \partial(\phi_2(\Omega)).
$$
Clearly, we can view $\gamma$ as a map $\gamma:[0,1]\to \partial\Omega\subset M_2-\phi_2(\Omega)$ with $\gamma(0)=\gamma(1)=x_0$. Let
$$p_2:\hat M_2\to M_2-\phi_2(\Omega)$$
be the universal covering of $M_2-\phi_2(\Omega)$. Since $\gamma$ is homotopically non-trivial in $M_2-\phi_2(\Omega)$, any lift $\hat\gamma:[0,1]\to \hat M_2$ of $\gamma$ must satisfy $\hat\gamma(0)\neq \hat\gamma(1)$. If we can extend the covering $p_2$ to a covering
$$p:\hat M\to(\mathbb T^n,i)\sharp(M_2,\phi_2),$$
then $\hat\gamma$ is also a lift of $\gamma$ in $\hat M$ and the fact $\hat\gamma(0)\neq \hat\gamma(1)$ yields that $\gamma$ is homotopically non-trivial in $(\mathbb T^n,i)\sharp(M_2,\phi_2)$. Clearly the boundary map $\partial p_2:\partial\hat M_2\to \partial \phi_2(\Omega)$ is also a covering map. Since $\partial \phi_2(\Omega)$ is locally path-connected, so is $\partial \hat M_2$ and then $\partial\hat M_2$ is homeomorphic to the disjoint union of all its components with induced topology. That is ,we can write
$$
\partial\hat M_2=\bigsqcup_{\alpha\in \mathcal I} C_\alpha,
$$
where $\{C_\alpha\}_{\alpha\in\mathcal I}$ collects all components of $\partial \hat M_2$.

To define the desired extension, we make a discussion for the following two cases.

{\it Case 1. The dimension is three}.  Applying Corollary \ref{Cor: existence for covering} to each component $C_\alpha$, we can find a covering $p_{1,\alpha}:\hat M_{1,\alpha}\to \mathbb T^{3}-\Omega$ such that the boundary $\partial \hat M_{1,\alpha}$ has two connected components $\hat\partial_{1,\alpha}$ and $\hat\partial_{2,\alpha}$ such that there are homeomorphisms $\phi_{i,\alpha}:\hat\partial_{i,\alpha}\to  C_\alpha$ with
$$\partial p_{1,\alpha}|_{\hat\partial_{i,\alpha}}=\partial p_2\circ\phi_{i,\alpha} \quad\text{ for}\quad i=1,2.$$
Define
$$
\hat M_1=\bigsqcup_{\alpha\in \mathcal I}\hat M_{1,\alpha}
$$
and
$$
\hat\partial_i=\bigsqcup_{\alpha\in\mathcal I}\partial_{i,\alpha},\quad i=1,2.
$$
Denote
$$
\phi_i:\hat\partial_i\to \partial\hat M_2,\quad x\mapsto \phi_{i,\alpha}(x)\quad\text{if}\quad x\in \hat\partial_{i,\alpha}.
$$
Then $\phi_i$ defines a homeomorphism between $\hat\partial_i$ and $\partial\hat M_2$. Let
$$
\hat M=\left(\hat M_1\sqcup_{\phi_1}\hat M_2\right)\sqcup_{\phi_2}\hat M_2.
$$
Then we can glue the maps $p_1$ and $p_2$ on $\hat M$ to obtain a covering $p:\hat M\to (\mathbb T^n,i)\sharp(M_2,\phi_2)$, which serves as a desired extention for $p_2:\hat M_2\to M_2-\phi_2(\Omega)$.

{\it Case 2. The dimension is greater than three}. The argument is similar to that in case 1 except some slight modifications. For each $C_\alpha$ we use Lemma \ref{Lem: existence for covering D4+} to construct a covering $p_{1,\alpha}:\hat M_{1,\alpha}\to \mathbb T^{n}-\Omega$ such that the boundary is homeomorphic $C_\alpha$ by homeomorphism  $\phi_{\alpha}:\partial\hat M_{1,\alpha}\to  C_\alpha$ satisfying
$$\partial p_{1,\alpha}|_{\partial\hat M_{1,\alpha}}=\partial p_2\circ\phi_{\alpha}.$$
Then we can define $\hat M_1$ in the same way such that there is a homeomorphism $\phi:\partial\hat M_1\to\partial\hat M_2$. Now the desired extension is given by $\hat M_1\sqcup_\phi\hat M_2$.
\end{proof}

\bibliography{bib}
\bibliographystyle{amsplain}
\end{document}